\documentclass[10pt]{article}

\usepackage{amssymb,amsthm,amsmath,hyperref}

\usepackage{graphicx,float}

\numberwithin{equation}{section}

\newtheorem{thm}{Theorem}[section]
\newtheorem{lem}{Lemma}[section]

\newtheorem{prop}{Proposition}[section]
\theoremstyle{definition}
\newtheorem{defn}{Definition}[section]
\theoremstyle{remark}
\newtheorem{rem}{Remark}[section]

\def\supetage#1#2{
\sup_{\scriptstyle {#1}\atop\scriptstyle {#2}} }

\allowdisplaybreaks

\begin{document}
\title{Global solutions of compressible Navier--Stokes equations with a density--dependent viscosity coefficient}
\author{Ting Zhang\thanks{E-mail: zhangting79@hotmail.com}\\
\textit{\small Department of Mathematics, Zhejiang University,
Hangzhou 310027, China} }
\date{}
\maketitle
\begin{abstract}
  We prove the global existence and uniqueness of the classical (weak) solution for the 2D or 3D compressible
  Navier--Stokes equations with a density--dependent viscosity
  coefficient ($\lambda=\lambda(\rho)$).
  Initial data and solutions are only small in the energy-norm.
  We  also give a description of the large time
  behavior of the solution. Then, we study the propagation of singularities in solutions. We obtain that
  if there is a vacuum domain at
  initially, then the vacuum domain will exists for all time, and
  vanishes as time goes to infinity.
\end{abstract}

\section{Introduction}
In this paper, we consider the following compressible Navier--Stokes
equations
    \begin{equation}
    \left\{
    \begin{array}{l}
        \rho_t+\mathrm{div}(\rho u)=0,\\
            (\rho u)_t+\mathrm{div}(\rho u\otimes u)+\nabla P=
            \mu\Delta
            u+\nabla((\mu+\lambda(\rho))\mathrm{div}u)+\rho f,
    \end{array}
    \right.\label{3VDD-full-E1.1}
    \end{equation}
for $x\in\mathbb{R}^N$ and $t>0$, $N=2$ or $3$, with the boundary
and initial conditions
    \begin{equation}
      u(x,t)\rightarrow0,
      \ \rho(x,t)\rightarrow \tilde{\rho}>0, \ \textrm{ as
      }|x|\rightarrow\infty, \ t>0,
    \end{equation}
    \begin{equation}
      (\rho,u)|_{t=0}=(\rho_0,u_0).\label{3VDD-full-E1.2}
    \end{equation}
Here $\rho({x},t)$, $u({x},t)$ and $P=P(\rho)$ stand for the fluid
density, velocity and pressure respectively, $f$ is a given external
force, the dynamic viscosity coefficient $\mu$ is a positive
constant, the second viscosity coefficient $\lambda=\lambda(\rho)$
is a function of $\rho$.

In \cite{Zhang2008}, we proved the global existence of weak
solutions for the two-dimensional system, and study the propagation
of singularities in solutions. In this paper, we want to obtain the
global existence, uniqueness and the large time behavior of the
classical solution to the system
(\ref{3VDD-full-E1.1})--(\ref{3VDD-full-E1.2}) in $\mathbb{R}^2$ or
$\mathbb{R}^3$,  also obtain the global existence of weak solutions
and study the propagation of singularities in solutions in
$\mathbb{R}^3$.

At first, we obtain the global existence,  uniqueness and the large
time behavior of the classical solution, when the energy of initial
data is small, but the oscillation is arbitrarily large.
Specifically, we fix a positive constant $\tilde{\rho}$, assume that
$(\rho_0-\tilde{\rho},u_0)$ are small in $L^2$,  and
$\rho_0-\tilde{\rho},u_0\in
 H^3$ with no restrictions on their norms, (since we use the classical analysis methods in this paper, we
restrict the result of the existence of the classical solutions on
the framework of Hilbert space $H^{3}(\mathbb{R}^{N})\hookrightarrow
C^1(\mathbb{R}^{N})$). Our existence result accommodates a wide
class of pressures $P$, including pressures that are not monotone in
$\rho$.
  It also generalizes and improves upon earlier results of Danchin
 \cite{Danchin2004} and
Matsumura-Nishida \cite{Marsumura83}
 in a significant way:
$(\rho_0-\tilde{\rho},u_0)$ are only small  in $L^2$.

Now, we give a precise formulation of our result. Concerning the
pressure $P$, viscosity coefficients $\mu$ and $\lambda$, we fix
$0<\tilde{\rho}<\bar{\rho}$ and assume that
    \begin{equation}
      \left\{
      \begin{array}{l}
       P\in C^1([0,\bar{\rho}])\cap C^3((0,\bar{\rho}]),\ \lambda\in C^1([0,\bar{\rho}])
      \cap C^3((0,\bar{\rho}]),\\
       \mu>0,
      \ \lambda\in\left\{\begin{array}{ll}
        [0,\infty),& N=2,\cr
        [0,3\mu), & N=3,
      \end{array}
      \right. \textrm{for all }\rho\in[0,\bar{\rho}],\\
           P(0)=0,
              \ P'(\tilde{\rho})>0, \\
              (\rho-\tilde{\rho})[P(\rho)-P(\tilde{\rho})]>0,
                \  \rho\in[0,\tilde{\rho})\cup(\tilde{\rho},\bar{\rho}].
      \end{array}
      \right.\label{3VDD-full-E1.5}
    \end{equation}
Let $G$ be the potential energy density, defined by
    \begin{equation}
      G(\rho)=\rho\int^{\rho}_{\tilde{\rho}}\frac{P(s)-P(\tilde{\rho})}{s^2}ds.
    \end{equation}
Then, for any $g\in C^2([0,\bar{\rho}])$ with
$g(\tilde{\rho})=g'(\tilde{\rho})=0$, there is a constant $C$ such
that $|g(\rho)|\leq CG(\rho)$, $\rho\in[0,\bar{\rho}]$.

Define
        \begin{equation}
          C_0=\int\left(\frac{1}{2}\rho_0|u_0|^2+G(\rho_0)\right)dx,\label{3VDD-full-E1.6}
        \end{equation}
            \begin{equation}
              C_f=\sup_{t\geq0}\|f(\cdot,t)\|_{L^2}^2+
              \int^\infty_0\left(
                \|f(\cdot,t)\|_{L^2}+\|f(\cdot,t)\|_{L^2}^2
                +\sigma^{3N-2}\|\nabla f\|_{L^4}^4+\sigma^{2N-1}\|
                f_t(\cdot,t)\|_{L^2}^2
              \right)dt
            \end{equation}
 and
        \begin{equation}
          M_q=\int^\infty_0\left(
      \sigma^2\| f_t\|_{L^2}^2+\sigma^3\|\nabla
      f\|_{L^4}^4+\sigma^{p_1}\|f_t\|_{L^{2+q}}^{2+q}
      +\sigma^{p_1}\|\nabla
      f\|_{L^{p_2}}^{p_2}
      \right)dt+\|\nabla u_0\|_{L^2}^2+\sup_{t\geq0}\|f\|_{L^{2+q}},\label{3VDD-full-E1.10}
        \end{equation}
     where  $\sigma(t)=\min\{1,t\}$,
        \begin{equation}
          p_1=\left\{\begin{array}{ll}
                2+\frac{q}{2}, &N=2,\\
                    1+\frac{5q}{4}, &N=3,
          \end{array}
          \right.\ p_2=\left\{\begin{array}{ll}
                4+2q, &N=2,\\
                   \frac{ 6(2+q)}{4-q}, &N=3,
          \end{array}
          \right.
        \end{equation}
      and $q$ is
     a constant satisfying
        \begin{equation}
          q\in\left\{
          \begin{array}{ll}
            (0,2), &N=2,\\
          (1,\frac{4}{3}), & N=3,
          \end{array}
          \right.\ \textrm{ and }
          q^2<\frac{4\mu}{\mu+\lambda(\rho)},
          \ \forall\ \rho\in[0,\bar{\rho}].
        \end{equation}
As in \cite{Hoff2008,Hoff2005}, we recall the definition of the
vorticity matrix $w^{j,k}=\partial_k u^j-\partial_j u^k$, and
definition of the function
    \begin{equation}
      F=(\lambda+2\mu)\mathrm{div}u-P(\rho)+P(\tilde{\rho}).
    \end{equation}
Thus, we have
\begin{equation}
    \Delta
    u^j=\partial_j(\frac{F+P-P(\tilde{\rho})}{\lambda+2\mu})+\partial_i(w^{j,i}).
    \label{3VDD-full-E2.24}
    \end{equation}
 We also define the convective derivative
$\frac{D}{Dt}$ by $\frac{Dw}{Dt}=\dot{w}=w_t+u\cdot\nabla w$,  the
H\"{o}lder norm
        $$
        <v>^\alpha_A=\supetage{x,y\in A}{x\neq
        y}\frac{|v(x)-v(y)|}{|x-y|^\alpha},
        $$
and
        $$
        <g>^{\alpha,\beta}_{A\times[t_1,t_2]}=\supetage{(x,t),(y,s)\in A\times[t_1,t_2]}{(x,t)\neq
        (y,s)}\frac{|g(x,t)-g(y,s)|}{|x-y|^\alpha+|t-s|^\beta},
        $$
where  $v:A\subseteq \mathbb{R}^N\rightarrow \mathbb{R}^N$,
 $g:A\times[t_1,t_2]\rightarrow \mathbb{R}^N$ and
 $\alpha,\beta\in(0,1]$.

The following is the main result of this paper.

\begin{thm}\label{3VDD-full-T1.1}
  Assume that conditions (\ref{3VDD-full-E1.5})--(\ref{3VDD-full-E1.10}) hold. Then,
  for a
  given  positive number $M$ (not necessarily small) and
  $\bar{\rho}_1\in(\tilde{\rho},\bar{\rho})$, there are positive
  number $\varepsilon$,
   such that,  the Cauchy problem
  (\ref{3VDD-full-E1.1})--(\ref{3VDD-full-E1.2}) with
  the  initial data $(\rho_0,u_0)$ and external force $f$
  satisfying
    \begin{equation}
      \left\{
      \begin{array}{l}
       0<\underline{\rho}_1\leq \inf \rho_0\leq\sup\rho_0\leq\bar{\rho}_1,\\
            C_0+C_f\leq\varepsilon,\\
                M_q\leq M,\\
                \rho_0-\tilde{\rho},\ u_0\in H^3,\ f_t\in
                C([0,\infty);L^2),\ f\in
                C([0,\infty);H^2),
      \end{array}
      \right.
    \end{equation}
 has a unique
global classical solution $(\rho,u)$  satisfying
    \begin{equation}
      \frac{1}{2}\underline{\rho}_1\leq\rho(x,t)\leq\bar{\rho},\
      \ x\in\mathbb{R}^N,\ t\geq0,\label{3VDD-full-E1.15}
    \end{equation}
     \begin{equation}
      (\rho-\tilde{\rho},u)\ \in C^1(\mathbb{R}^N\times[0,T])\cap C([0,T],H^{3})\cap
C^{1}([0,T],H^{2}),
    \end{equation}
        \begin{equation}
    \sup_{t\in[0,T]}\left(\|(\rho-\tilde{\rho},u)\|_{H^3}+\|(\rho_t,u_t)\|_{H^2}
    \right)+\int^{T}_0\|u\|_{H^4}^2dt\leq K(T),\ \forall\ T>0,
        \end{equation}
                   \begin{equation}
     \lim_{t\rightarrow+\infty}\int\left(|\rho-\tilde{\rho}|^4+\rho|u|^4\right)(x,t)dx=0,
     \label{3VDD-full-E1.15-10}
    \end{equation}
     \begin{equation}
    \lim_{t\rightarrow+\infty}\int
    |\nabla u|^2(x,t)dx=0,
     \label{3VDD-full-E1.15-20}
    \end{equation}
where $K(T)$ is a positive constant dependent on
$\underline{\rho}_1$, $\bar{\rho}_1$,
$\|(\rho_0-\tilde{\rho},u_0)\|_{H^3}$ and $T$.
\end{thm}

\begin{rem}
  For example, we can choose that $P=A\rho^\gamma$ and
  $\lambda(\rho)=c\rho^\beta$ with $\gamma\geq 1$ and
  $\beta\geq1$, where $A$ and $c$ are two positive constants.
  Also, we can choose that $\lambda$ is a non-negative constant.
\end{rem}

\begin{rem}
Considering the case that the space domain
$\Omega\subset\mathbb{R}^2$ is bounded and $\lambda=\rho^\beta$,
$\beta>3$, Vaigant-Kazhikhov (\cite{Vaigant}, Theorems 1-2) obtained
the global existence of strong solutions when the initial data are
large and the initial density is bounded from zero. In this paper,
since the initial energy is small, we can use the similar argument
as that in the case $\lambda=$constant
\cite{Hoff2008,Hoff2005,Hoff1995,Hoff1995-2} to obtain some good a
priori estimates of the solution, and obtained the global existence
of classical solutions when the space domain is $\mathbb{R}^N$,
$N=2$ or $3$, the initial density may vanish in an open set and
$\beta\geq1$.
\end{rem}

The proof of Theorem \ref{3VDD-full-T1.1} bases on the derivation of
a priori estimates for the local solution. Specifically, in Section
\ref{3VDD-full-S2}, we fix a smooth, local in time solution for
which $0\leq\rho\leq\overline{\rho}$ and $A_1+A_2\leq
2(C_0+C_f)^\theta$, then obtain the estimate $A_1+A_2\leq
(C_0+C_f)^\theta$, and prove that the density remains in a compact
subset of $(0,\overline{\rho})$. Using the classical continuation
method, we can close these estimates.

Using the initial condition $u_0\in H^1$, we can obtain pointwise
bounds for $F$ in Proposition \ref{3VDD-full-P2.8}, which is the key
point of the a priori estimates. Because that the mass equation can
be transformed to the following form,
    \begin{equation}
      \frac{d}{dt}\Lambda(\rho(x(t),t))+P(\rho(x(t),t))-P(\tilde{\rho})=-F(x(t),t),\label{3VDD-full-E1.18}
    \end{equation}
where $\Lambda$ satisfies that $\Lambda(\tilde{\rho})=0$ and
$\Lambda'(\rho)=\frac{2\mu+\lambda(\rho)}{\rho}$, a curve $x(t)$
satisfies $\dot{x}(t)=u(x(t),t)$, thus pointwise bounds for the
density will therefore follow from  pointwise bounds for $F$.

In theorem \ref{2VDD-T1.1}, the constant $\varepsilon$ is
independent of $\underline{\rho}_1$.
  Thus, we can obtain the global existence of weak solutions to (\ref{3VDD-full-E1.1})--(\ref{3VDD-full-E1.2})
 with the nonnegative initial density $\rho_0\geq0$ (the two-dimensional result can be found in \cite{Zhang2008}).
\begin{defn}
We say that $(\rho,u)$ is a weak solution of
(\ref{3VDD-full-E1.1})--(\ref{3VDD-full-E1.2}), if $\rho$ and $u$
are suitably integrable and satisfy that
\begin{itemize}
    \item
        \begin{equation}
          \left.\int\rho\phi
          dx\right|^{t_2}_{t_1}=\int^{t_2}_{t_1}\int(\rho\phi_t+\rho
          u\cdot\nabla\phi)dxdt\label{3VDD-full-E1.3}
        \end{equation}
        for all times $t_2\geq t_1\geq0$ and all $\phi\in
        C^1_0(\mathbb{R}^N\times[t_1,t_2])$,
    \item
      \begin{eqnarray}
          &&\left.\int\rho u\psi
          dx\right|^{t_2}_{t_1}-\int^{t_2}_{t_1}\int\{\rho
          u\cdot\psi_t+\rho(u\cdot\nabla \psi
          )\cdot u+P\mathrm{div}\psi\}dxdt\nonumber\\
                &=&-\int^{t_2}_{t_1}\int
          \{\mu \partial_k
          u^j\partial_k\psi^j+(\mu+\lambda)\mathrm{div}u\mathrm{div}\psi
          -\rho f\psi\}dxdt\label{3VDD-full-E1.4}
        \end{eqnarray}
        for all times $t_2\geq t_1\geq0$ and all $\psi\in
        (C^1_0(\mathbb{R}^N\times[t_1,t_2]))^N$.
\end{itemize}
\end{defn}
Concerning the pressure $P$, viscosity coefficients $\mu$ and
$\lambda$, we fix $0<\tilde{\rho}<\bar{\rho}$ and assume that
\begin{equation}
      \left\{
      \begin{array}{l}
      P\in C^1([0,\bar{\rho}]),\ \lambda\in C^2([0,\bar{\rho}]),\\
       \mu>0,
      \ \ \lambda\in\left\{\begin{array}{ll}
        [0,\infty),& N=2,\cr
        [0,3\mu), & N=3,
      \end{array}
      \right. \textrm{for all }\rho\in[0,\bar{\rho}],\\
           P(0)=0,
              \ P'(\tilde{\rho})>0, \\
              (\rho-\tilde{\rho})[P(\rho)-P(\tilde{\rho})]>0,
                \ \rho\in[0,\tilde{\rho})\cup(\tilde{\rho},\bar{\rho}],\\
              P\in C^2([0,\bar{\rho}])
              \ \textrm{ or }\    \frac{P(\cdot)}{2\mu+\lambda(\cdot)}
              \ \textrm{ is a monotone function on}\ [0,\bar{\rho}].
      \end{array}
      \right.\label{2VDD-E1.5}
    \end{equation}
\begin{thm}\label{2VDD-T1.1}
  Assume that conditions (\ref{3VDD-full-E1.6})--(\ref{3VDD-full-E1.10}) and (\ref{2VDD-E1.5}) hold. Then,
  for a
  given  positive number $M$ (not necessarily small) and
  $\bar{\rho}_1\in(\tilde{\rho},\bar{\rho})$, there are positive
  numbers $\varepsilon$
  and  $\theta$, such that,  the Cauchy problem
  (\ref{3VDD-full-E1.1})--(\ref{3VDD-full-E1.2}) with
  the  initial data $(\rho_0,u_0)$ and external force $f$
  satisfying
    \begin{equation}
      \left\{
      \begin{array}{l}
       0\leq \inf \rho_0\leq\sup\rho_0\leq\bar{\rho}_1,\\
            C_0+C_f\leq\varepsilon,\\
                M_q\leq M,
      \end{array}
      \right.
    \end{equation}
 has a
global weak solution $(\rho,u)$ in the sense of
(\ref{3VDD-full-E1.3})--(\ref{3VDD-full-E1.4}) satisfying
    \begin{equation}
       C^{-1}\inf\rho_0\leq\rho\leq\bar{\rho},\
      \textrm{a.e.}\label{2VDD-E1.15}
    \end{equation}
        \begin{equation}
          \left\{
          \begin{array}{l}
            \rho-\tilde{\rho},\ \rho u\in C([0,\infty);H^{-1}),\\
                \nabla u\in L^2(\mathbb{R}^N\times[0,\infty)),
          \end{array}
          \right.
        \end{equation}
        \begin{equation}
          <u>^{\alpha,\frac{\alpha}{N+2\alpha}}_{\mathbb{R}^N\times[\tau,\infty)}+
          \sup_{t\geq \tau}\left(\|\nabla F(\cdot,t)\|_{L^2}+\|\nabla
          w(\cdot,t)\|_{L^2}\right)
          \leq C(\tau)(C_0+C_f)^\theta,\label{2VDD-E1.15-1}
        \end{equation}
         \begin{equation}
          \sup_{t>0}\int|\nabla u|^2dx+\int^\infty_0\int\sigma
          |\nabla\dot{u}|^2dxds\leq C,\label{2VDD-E1.15-2}
        \end{equation}
    \begin{equation}
      \int^T_0\|F(\cdot,t)\|_{L^\infty}dt\leq C(T),
    \end{equation}
    \begin{equation}
          <F>^{{\alpha'},\frac{{\alpha'}}{N+2{\alpha'}}}_{\mathbb{R}^N\times[\tau,T]}+
     <w>^{{\alpha'},\frac{{\alpha'}}{N+2{\alpha'}}}_{\mathbb{R}^N\times[\tau,T]}\leq
    C(\tau,T),\label{2VDD-E1.22}
    \end{equation}
where $\alpha\in(0,1)$ when $N=2$, $\alpha\in(0,\frac{1}{2}]$ when
$N=3$, $\tau >0$ and ${\alpha'}\in(0,\frac{2+q-N}{2+q}]$,
    \begin{eqnarray}
      &&\sup_{t>0}\int\left(
      \frac{1}{2}\rho|u|^2+|\rho-\tilde{\rho}|^2
      +\sigma|\nabla u|^2      \right)dx\nonumber\\
            &&+\int^\infty_0\int\left(
            |\nabla u|^2+\sigma|(\rho u)_t+\mathrm{div}(\rho u\otimes
            u)|^2+\sigma^N|\nabla\dot{u}|^2     \right)dxdt\nonumber\\
      &\leq& (C_0+C_f)^\theta,\label{2VDD-E1.17}
    \end{eqnarray}
                  \begin{equation}
    \lim_{t\rightarrow+\infty}\int\left(|\rho-\tilde{\rho}|^4+\rho|u|^4\right)(x,t)dx=0.
    \label{2VDD-full-E1.28}
    \end{equation}
In addition, in the case that $\inf\rho_0>0$, the term
$\int^\infty_0\int\sigma|\dot{u}|^2dxdt$ may be included on the left
hand side of (\ref{2VDD-E1.17}).
\end{thm}

\begin{rem}
  Considering the case that $\lambda=$constant, Hoff-Santos
\cite{Hoff2008} and Hoff \cite{Hoff2005,Hoff1995,Hoff1995-2}
obtained the existence of global weak solutions. In this paper,
since the viscosity
    coefficient $\lambda$ is a function of the density $\rho$,     we
need a higher regularity condition $\nabla u_0\in L^2$, and use some
new methods to obtain a priori estimates of the solution. Using the
initial condition $\nabla u_0\in L^2$, we can obtain pointwise
bounds for $F$ in Proposition \ref{3VDD-full-P2.8}, which is the key
point of the a priori estimates. Using the compensated compactness
method \cite{Vaigant} and the estimate
$\int^T_0\|F(\cdot,t)\|_{L^\infty}dt\leq C(T)$, we can obtain the
strong limit of approximate densities $\{\rho^\delta\}$, see Section
\ref{2VDD-S3}.
\end{rem}

Then, we study the propagation of singularities in solutions
obtained in Theorem \ref{2VDD-T1.1}. Under the regularity estimates
of the solution in Theorem \ref{2VDD-T1.1}, and similar arguments as
that in \cite{Hoff2008,Zhang2008}, we can obtain Theorems
\ref{2VDD-T1.2}--\ref{2VDD-T1.6}, and omit the details.

In Theorem \ref{2VDD-T1.2}, we obtain that each point of
$\mathbb{R}^N$ determines a unique integral curve of the velocity
field at the initial time $t=0$, and that this system of integral
curves defines a locally bi-H\"{o}lder homeomorphism of any open
subset $\Omega$ onto its image $\Omega^t$ at each time $t>0$. From
this Lagrangean structure,  we can obtain that if there is a vacuum
domain at the
  initial time, then the vacuum domain will exist for all time, and
  vanishes as time goes to infinity, see Theorem \ref{2VDD-T1.4}.
 Also, in Theorem \ref{2VDD-T1.5}, we obtain that, if
the initial density has a limit at a point from a given side of a
continuous hypersurface, then at each later time both the density
and the divergence of the velocity have limits at the transported
point from the corresponding side of the transported hypersurface,
which is also a continuous manifold. If the limits from both sides
exist, then the Rankine-Hugoniot conditions hold in a strict
pointwise sense, showing that the jump in the
$(\lambda+2\mu)\mathrm{div}u$ is proportional to the jump in the
pressure (Theorem \ref{2VDD-T1.6}). This leads to a derivation of an
explicit representation for the strength of the jump in
$\Lambda(\rho)$ in non-vacuum domain.

\begin{thm}\label{2VDD-T1.2}
Assume that the conditions of Theorem \ref{2VDD-T1.1} hold.
\begin{description}
    \item[(1)] For each $t_0\geq0$ and $x_0\in\mathbb{R}^N$, there is a unique curve
    $X(\cdot;x_0,t_0)\in
    C^1((0,\infty);\mathbb{R}^N)\cap C^{\frac{6-N-2\alpha}{4}}([0,\infty);\mathbb{R}^N)
    $,  satisfying
        \begin{equation}
          X(t;x_0,t_0)=x_0+\int^t_{t_0} u(X(s;x_0,t_0),s)ds.\label{2VDD-E1.26}
        \end{equation}
    \item[(2)] Denote $X(t,x_0)=X(t;x_0,0)$. For each $t>0$ and any open set $\Omega\subset\mathbb{R}^N$,
    $\Omega^t=X(t,\cdot)\Omega$ is open and the map $x_0\longmapsto
    X(t,x_0)$ is a homeomorphism of $\Omega$ onto $\Omega^t$.
    \item[(3)] For any $0\leq t_1,t_2\leq T$, the map $X(t_1,y)\rightarrow
    X(t_2,y)$ is H\"{o}lder continuous from $\mathbb{R}^N$ onto
    $\mathbb{R}^N$. Specifically, for any $y_1,y_2\in
    \mathbb{R}^N$,
        \begin{equation}
    |X(t_2,y_2)-X(t_2,y_1)|\leq \exp(1-e^{-C (1+T)})|X(t_1,y_2)-X(t_1,y_1)|^{e^{-C
    (1+T)}}.
        \end{equation}
    \item[(4)] Let $\mathcal{M}\subset\mathbb{R}^N$ be a $C^\beta$
    $(N-1)$-manifold, where $\beta\in[0,1)$. Then, for any $t>0$,
    $\mathcal{M}^t=X(t,\cdot)\mathcal{M}$ is a $C^{\beta'}$
    $(N-1)$-manifold, where $\beta'=\beta e^{-C
    (1+t)}$.
\end{description}
\end{thm}

\begin{thm}\label{2VDD-T1.3}
Assume that the conditions of Theorem \ref{2VDD-T1.1} hold. Let $V$
be a nonempty open set in $\mathbb{R}^N$. If
$\mathrm{essinf}\rho_0|_V\geq \underline{\rho}>0$, then there is a
positive number $\underline{\rho}^{-}$ such that,
    $$
    \rho(\cdot,t)|_{V^t}\geq \underline{\rho}^-,
    $$
for all $t>0$, where $V^t=X(t,\cdot)V$.
\end{thm}

\begin{thm}\label{2VDD-T1.4}
Assume that the conditions of Theorem \ref{2VDD-T1.1} hold. Let $U$
be a nonempty open set in $\mathbb{R}^N$. Assume that $\rho_0|_U
=0$. Then,
    $$
    \rho(\cdot,t)|_{U^t}=0,
    $$
for all $t>0$, where $U^t=X(t,\cdot)U$. Furthermore, we have
       \begin{equation}
    \lim_{t\rightarrow\infty}|\{x\in\mathbb{R}^N|\rho(x,t)=0\}|=0.\label{2VDD-E1.29}
    \end{equation}
\end{thm}

Recall that  the oscillation of $g$  at $x$ with respect to $E$ is
defined by (as in \cite{Hoff2008})
    $$
    \mathrm{osc}(g;x,E)=\lim_{R\rightarrow0}
    \left(\mathrm{esssup}g|_{E\cap B_R(x)}-\mathrm{essinf}g|_{E\cap
    B_R(x)}
    \right),
    $$
where $x\in \bar{E}$ and $g$ maps an open set $E\subset
\mathbb{R}^N$ into $\mathbb{R}$. We shall say that $g$ is continuous
at an interior point $x$ of $E$, if osc$(g;x,E)=0$.
\begin{thm}\label{2VDD-T1.5}
  Assume that the conditions of Theorem \ref{2VDD-T1.1} hold. Let
  $E\subset\mathbb{R}^N$ be open  and $x_0\in\bar{E}$. If
  $\mathrm{osc}(\rho_0;x_0,E)=0$, then
  $\mathrm{osc}(\rho(\cdot,t);X(t,x_0),X(t,\cdot)E)=0$. In
  particular, if $x_0\in E$ and $\rho_0$ is continuous at $x_0$,
  then $\rho(\cdot,t)$ is continuous at $X(t,x_0)$.
\end{thm}

Now, let $\mathcal{M}$ be a $C^0$ $(N-1)$-manifold in $\mathbb{R}^N$
and
 $x_0\in\mathcal{M}$. Then there is a neighborhood $G$ of $x_0$
which is the disjoint union $G=(G\cap\mathcal{M})\cup E_+\cup E_-$,
where $E_{\pm}$ are open and $x_0$ is a limit point of each. If
osc$(g;x_0,E_+)=0$, then the common value $g(x_0+,t)$ is the
one-sided limit of $g$ at $x_0$ from the plus-side of $\mathcal{M}$,
and similar for the one-sided limit $g(x_0-,t)$ from the minus-side
of $\mathcal{M}$. If both of these limits exist, then the difference
$[g(x_0)]:=g(x_0+)-g(x_0-)$ is the jump in $g$ at $x_0$ with respect
to $\mathcal{M}$ (see \cite{Hoff2008}). Then, we can obtain the
following result about the propagation of singularities in
solutions.

\begin{thm}\label{2VDD-T1.6}
  Let $(\rho,u)$ as in Theorem \ref{2VDD-T1.1},
  $\mathcal{M}$ be a $C^0$ $(N-1)$-manifold and $x_0\in \mathcal{M}$.

  (a) If $\rho_0$ has a one-sided limit at $x_0$ from the
  plus-side of $\mathcal{M}$, then for each $t>0$, $\rho(\cdot,t)$
  and $\mathrm{div}u(\cdot,t)$ have one-sided limits at $X(t,x_0)$
  from the plus-side of the $C^0$ $(N-1)$-manifold
  $X(t,\cdot)\mathcal{M}$ corresponding to the choice
  $E_+^t=X(t,\cdot)E_+$. The map $t\mapsto \rho(X(t,x_0)+,t)$ is
  in $C^{\frac{1}{4}}([0,\infty))\cap C^1((0,\infty))$ and the map $t\mapsto \mathrm{div}u(X(t,x_0)+,t)$ is
  locally H\"{o}lder continuous on $(0,\infty)$.

  (b) If both one-sided limits $\rho_0(x_0\pm)$ of $\rho_0$ at
  $x_0$ with respect to $\mathcal{M}$ exist, then for each $t>0$,
  the jumps in $P(\rho(\cdot,t))$ and $\mathrm{div}u(\cdot,t)$ at
  $X(t,x_0)$ satisfy the Rankine-Hugonoit condition
        \begin{equation}
          [(2\mu+\lambda(\rho(X(t,x_0),t)))\mathrm{div}u(X(t,x_0),t)]
          =[P(\rho(X(t,x_0),t))].\label{2VDD-E1.31}
        \end{equation}

  (c) Furthermore, if $\rho_0(x_0\pm)\geq\underline{\rho}>0$, then the jump in
   $\Lambda(\rho)$ satisfies the representation
    \begin{equation}
      [\Lambda(\rho(X(t,x_0),t))]=\exp\left(
      -\int^t_0a(\tau,x_0)d\tau
      \right)[\Lambda(\rho_0(x_0))]
    \end{equation}
  where
  $a(t,x_0)=\frac{[P(\rho(X(t,x_0),t))]}{[\Lambda(\rho(X(t,x_0),t))]}$.
\end{thm}

\begin{rem}
Using similar arguments as that in \cite{Zhang2008}, we also can
show that the condition of $\mu=$constant will induce
   a singularity of the system at vacuum in the following
   two aspects: 1) considering the special case where two fluid regions initially
separated by a vacuum region, the solution we obtained is a
nonphysical weak solution in which separate
   kinetic energies of the two fluids need not to be conserved; 2)
    smooth solutions for the spherically
symmetric system will blowup when the initial density is compactly
supported. Thus, the viscosity coefficient $\mu$ plays a key role in
the Navier-Stokes equations.
\end{rem}

We now briefly review some previous works about the Navier--Stokes
equations with density--dependent viscosity coefficients. For the
free boundary problem of one-dimensional or spherically symmetric
isentropic fluids,  there are many works, please see
\cite{jiang,Li2008,liu,yang3,fang00,zhang1}  and the references
cited therein. Under a special condition between $\mu$ and
$\lambda$, $\lambda=2\rho\mu'-2\mu$, there are some existence
results of global weak solutions for the system with the Korteweg
stress tensor or the additional quadratic friction term, see
\cite{bresch1,bresch2}. Also see Lions\cite{Lions} for
multidimensional isentropic fluids.

  We should mention that the methods
    introduced by    Hoff in
    \cite{Hoff1995} and Vaigant-Kazhikhov in \cite{Vaigant}
   will play a crucial role
in our proof here.

\section{Global existence}\label{3VDD-full-S2}
Standard local existence results now apply to show that there is a
smooth local solution $(\rho,u)$ to
(\ref{3VDD-full-E1.1})--(\ref{3VDD-full-E1.2}), defined up to a
positive time $T_0$, such that
    $$
    \rho,\ u\in C^{1}(\mathbb{R}^{N}\times[0,T_{0}] )\ \ \
\mbox{with}\ \ \ \rho>0\ \ \mbox{for all}\ t\in[0,T_{0}]
    $$
     and
$$(\rho-\tilde{\rho},u)\ \in C([0,T_{0}],H^{3})\cap
C^{1}([0,T_{0}],H^{2}).$$ (See for example
Matsumura-Nishida\cite{Marsumura80} and  Nash\cite{Nash62}.) Let
$[0,T^{*})$ be the maximal existence interval of the above solution
to (\ref{3VDD-full-E1.1})--(\ref{3VDD-full-E1.2}).

In this section, we derive some a priori estimates for the local
smooth solution of the system
(\ref{3VDD-full-E1.1})--(\ref{3VDD-full-E1.2}).

\textbf{Claim 1:} For any $T>0$, if  $(\rho,u)$  satisfies
        \begin{equation}
      0\leq  \rho\leq\bar{\rho}\label{3VDD-full-E2.1-1}
        \end{equation}
and
    \begin{equation}
      A_1+A_2\leq 2(C_0+C_f)^\theta,\ \forall\ t\in[0,T]\cap[0,T^*),\label{3VDD-full-E2.1-2}
    \end{equation}
where $\theta\in(0,1)$,
    $$
    A_1(T)=\sup_{t\in(0, T]\cap(0,T^*)}\sigma\int|\nabla
    u|^2dx+\int^{T\wedge T^*}_0\int\sigma\rho|\dot{u}|^2dxdt
    $$
and
    $$
    A_2(T)=\sup_{t\in(0, T]\cap(0,T^*)}\sigma^N
    \int\rho|\dot{u}|^2dx+\int^{T\wedge T^*}_0\int\sigma^N|\nabla\dot{u}|^2dxdt,
    \ T\wedge T^*=\min\{T,T^*\},
    $$
then we have
        \begin{equation}
      \frac{1}{2}\underline{\rho}_1< \rho<\bar{\rho}\ \
      A_1+A_2\leq (C_0+C_f)^\theta, \ \forall\ t\in[0,T]\cap[0,T^*).
    \end{equation}
 In this paper, we assume that $\varepsilon\leq 1$.

We can rewrite the momentum equation in the form,
  \begin{equation}
    \rho\dot{u}^j=\partial_j F+\mu \partial_k w^{j,k}+\rho
    f^j.\label{3VDD-full-E1.19}
    \end{equation}
 Stated differently, the
decomposition (\ref{3VDD-full-E1.19}) implies that
        \begin{equation}
          \Delta F=\mathrm{div}(\rho\dot{u}-\rho f).\label{3VDD-full-E1.20}
        \end{equation}
Similarly, we have
        \begin{equation}
    \mu\Delta w^{j,k}=\partial_k(\rho \dot{u}^j)
    -\partial_j(\rho\dot{u}^k)+\partial_j(\rho f^k)-\partial_k(\rho
    f^j).\label{3VDD-full-E2.24-10}
    \end{equation}
    Thus $L^2$ estimates for $\rho\dot{u}$, immediately imply $L^2$
hounds for $\nabla F$ and $\nabla w$.
 These three relations
(\ref{3VDD-full-E1.19})--(\ref{3VDD-full-E2.24-10}) will play the
important role in this section.

From now on, the constant $C$ (or $C(T)$) will be independent of
$\underline{\rho}_1$.

\begin{prop}\label{3VDD-full-P2.1}
  There is a positive constant $C=C(\bar{\rho})$ independent of $\underline{\rho}_1$, such that if
  $(\rho,u)$ is a smooth solution of (\ref{3VDD-full-E1.1})--(\ref{3VDD-full-E1.2})
  satisfying (\ref{3VDD-full-E2.1-1})--(\ref{3VDD-full-E2.1-2}), then
        \begin{equation}
          \sup_{t\in[0,T]\cap[0,T^*)}\int\left[
          \frac{1}{2}\rho|u|^2+G(\rho)
          \right]dx+\int^{T\wedge T^*}_0\int|\nabla u|^2dxdt\leq
          C(C_0+C_f).\label{3VDD-full-E2.1}
        \end{equation}
\end{prop}
\begin{proof}
  Using the energy estimate, we can easily obtain
  (\ref{3VDD-full-E2.1}), and omit the details.
\end{proof}

The following lemma contains preliminary versions of $L^2$ bounds
for $\nabla u$ and $\rho\dot{u}$.

\begin{lem}\label{3VDD-full-L2.1}
  If $(\rho,u)$
  is a smooth solution of (\ref{3VDD-full-E1.1})--(\ref{3VDD-full-E1.2}) as
  in Proposition \ref{3VDD-full-P2.1}, then there is a constant $C=C(\bar{\rho})$  independent of $\underline{\rho}_1$, such that
    \begin{equation}
      \sup_{t\in(0,T]\cap(0,T^*)}\sigma\int|\nabla
      u|^2dx+\int^{T\wedge T^*}_0\int\sigma\rho|\dot{u}|^2dxdt
            \leq C\left(C_0+C_f+O_1
            \right),\label{3VDD-full-E2.4}
    \end{equation}
where $O_1=\int^{T\wedge T^*}_0\int\sigma|\nabla u|^3dxdt$, and
    \begin{eqnarray}
      &&\sup_{t\in(0,T]\cap(0,T^*)}\sigma^N\int\rho|\dot{u}|^2dx
      +\int^{T\wedge T^*}_0\int\sigma^N\left(|\nabla \dot{u}|^2
      +\left|\frac{D}{Dt}\mathrm{div}u
      \right|^2\right)dxdt\nonumber\\
            &\leq& C\left(C_0+C_f+A_1(T)
            \right)+C\int^{T\wedge T^*}_0\int\sigma^N(|u|^4+|\nabla
            u|^4)dxdt.\label{3VDD-full-E2.5}
    \end{eqnarray}
\end{lem}
\begin{proof}
In \cite{Zhang2008} (Lemma 2.1), we obtain this lemma in
$\mathbb{R}^2$. Using the similar argument as that in
\cite{Zhang2008} (Lemma 2.1) and \cite{Hoff1995} (Lemma 2.1), we can
easily obtain
 this lemma in
$\mathbb{R}^3$ and omit the details.
\end{proof}

The following  lemmas will be applied to bound the higher order
terms occurring on the right hand sides of
(\ref{3VDD-full-E2.4})--(\ref{3VDD-full-E2.5}).

\begin{lem}\label{3VDD-full-L2.2}
  If $(\rho,u)$
  is a smooth solution of (\ref{3VDD-full-E1.1})--(\ref{3VDD-full-E1.2}) as
  in Proposition \ref{3VDD-full-P2.1}, then there is a constant $C=C(\bar{\rho})$  independent of $\underline{\rho}_1$, such that,
     \begin{equation}
      \|u\|_{L^p}\leq C_p\|u\|_{L^2}^\frac{2N-Np+2p}{2p}\|\nabla
      u\|_{L^2}^{\frac{Np-2N}{2p}},\ p\in\left\{\begin{array}{ll}
        [2,\infty),&N=2,\cr
            [2,6], &N=3.
      \end{array}
      \right.\label{3VDD-full-E2.16}
    \end{equation}
    \begin{equation}
      \|u\|_{L^p}^p\leq
    C_p(C_0+C_f)^\frac{2N-Np+2p}{4}\|\nabla
      u\|_{L^2}^{\frac{Np-2N}{2}}+C_p(C_0+C_f)^{\frac{2N-Np+2p}{2N}}\|\nabla
      u\|_{L^2}^p,\ p\in\left\{\begin{array}{ll}
        [2,\infty),&N=2,\cr
            [2,6], &N=3,
      \end{array}
      \right.\label{3VDD-full-E2.16-0}
    \end{equation}
        \begin{equation}
          \|\nabla u\|_{L^p}\leq C_p(\|F\|_{L^p}
          +\|w\|_{L^p}+\|P-P(\tilde{\rho})\|_{L^p}
          ),\ p\in(1,\infty),\label{3VDD-full-E2.16-1}
        \end{equation}
      \begin{equation}
        \|\nabla F\|_{L^p}+    \|\nabla w\|_{L^p}\leq C_p(\|\rho\dot{u}\|_{L^p}
    +\|f\|_{L^p}),\ p\in(1,\infty),\label{3VDD-full-E2.16-2}
    \end{equation}
         \begin{eqnarray}
        \|F\|_{L^p}+    \| w\|_{L^p}&\leq& C_p\big(\|\rho\dot{u}\|_{L^2}^{\frac{Np-2N}{2p}}
        (\|\nabla
        u\|_{L^2}^{\frac{2N-Np+2p}{2p}}+\|P-P(\tilde{\rho})\|_{L^2}^{\frac{2N-Np+2p}{2p}})\nonumber\\
                &&
    +\|\nabla u\|_{L^2}+\|f\|_{L^2}+\|P-P(\tilde{\rho})\|_{L^2}\big),\ p\in\left\{\begin{array}{ll}
        [2,\infty),&N=2,\cr
            [2,6], &N=3.
      \end{array}
      \right.\label{3VDD-full-E2.16-3}
    \end{eqnarray}
Also, for $0\leq t_1\leq t_2\leq T$, $p\geq2$ and $s\geq0$,
    \begin{eqnarray}
      &&\left.\int\sigma^s|\rho-\tilde{\rho}|^pdx\right|^{t_2}_{t_1}
      +C^{-1}\int^{t_2}_{t_1}\int\sigma^s|\rho-\tilde{\rho}|^pdxds\nonumber\\
        &\leq& C\left(
        \int^{\sigma(t_2)\vee
        t_1}_{t_1}\int\sigma^{s-1}|\rho-\tilde{\rho}|^pdxds
        +\int^{t_2}_{t_1}\int\sigma^s|F|^pdxds
        \right),\label{3VDD-full-E2.25}
    \end{eqnarray}
    \begin{equation}
      \int^{t_2}_{t_1}\int\sigma^s|\rho-\tilde{\rho}|^pdxds
      \leq C\left(\int^{t_2}_{t_1}\int\sigma^s|F|^pdxds
      +C_0+C_f
      \right).\label{3VDD-full-E2.16-4}
    \end{equation}
\end{lem}
\begin{proof}
Using the similar argument as that in \cite{Zhang2008} (Lemma 2.2)
and \cite{Hoff1995} (Lemma 2.3), we can easily obtain
 this lemma  and omit the details.
\end{proof}

To bound the higher order term $\sigma|\nabla u|^3$ occurring on the
right hand sides of (\ref{3VDD-full-E2.4}) in $\mathbb{R}^3$, we
need to obtain the estimate of $\|u\|_{H^1_x}$ near $t=0$.

\begin{lem}\label{3VDD-full-L2.3}
 If $N=3$, $u_0\in H^1$, $(\rho,u)$
  is a smooth solution of (\ref{3VDD-full-E1.1})--(\ref{3VDD-full-E1.2}) as
  in Proposition \ref{3VDD-full-P2.1}, then there is a positive constant
  $T_1$  independent of $\underline{\rho}_1$, such that
    \begin{equation}
      \sup_{t\in[0,T_1\wedge T]\cap[0,T^*)}\int_{\mathbb{R}^3}
      |\nabla
      u|^2dx+\int^{T_1\wedge T\wedge T^*}_0\int_{\mathbb{R}^3}\rho|\dot{u}|^2dxdt
        \leq C(1+M_q).\label{3VDD-full-E2.29}
    \end{equation}
\end{lem}
\begin{proof}
  Using a similar argument as that in the proof of
  (\ref{3VDD-full-E2.4}), we have
    $$
      \int_{\mathbb{R}^3}|\nabla
      u|^2dx+\int^{t}_0\int_{\mathbb{R}^3}\rho|\dot{u}|^2dxdt
            \leq C(C_0+C_f+M_q)+C\int^{t}_0\int_{\mathbb{R}^3}|\nabla u|^3dxds.
   $$
  From
(\ref{3VDD-full-E2.16-1}) and (\ref{3VDD-full-E2.16-4}), we have
    $$
\int^{t}_0\int_{\mathbb{R}^3}|\nabla u|^3dxds
            \leq C+C\int^{t}_0\int_{\mathbb{R}^3}(|F|^3+|w|^3)dxds.
    $$
  From
(\ref{3VDD-full-E2.16-2})--(\ref{3VDD-full-E2.16-3}) and
(\ref{3VDD-full-E2.16}), we obtain
    \begin{eqnarray}
        && \int_{\mathbb{R}^3}\left(|F|^3+|w|^3\right)dx\nonumber\\
                &\leq&
      C\left(\int_{\mathbb{R}^3} |F|^2dx
      \right)^\frac{3}{4}\left(\int_{\mathbb{R}^3} |\nabla F|^2dx
      \right)^{\frac{3}{4}}+C\left(\int_{\mathbb{R}^3} |w|^2dx
      \right)^\frac{3}{4}\left(\int_{\mathbb{R}^3} |\nabla w|^2dx
      \right)^{\frac{3}{4}}\nonumber\\
            &\leq&C\left(\int_{\mathbb{R}^3}(|\nabla u|^2+|\rho-\tilde{\rho}|^2)dx\right)^\frac{3}{4}
            \left(\int_{\mathbb{R}^3}\left(\rho|\dot{u}|^2+|f|^2
            \right)dx\right)^\frac{3}{4}.
    \end{eqnarray}
Thus, from Proposition \ref{3VDD-full-P2.1}, we have
     \begin{eqnarray*}
    &&\int_{\mathbb{R}^3}|\nabla
      u|^2dx+\int^{t}_0\int_{\mathbb{R}^3}\rho|\dot{u}|^2dxdt\nonumber\\
            &\leq& C(1+M_q)+Ct\sup_{s\in[0,t]}
            \|\nabla u(\cdot,s)\|_{L^2}^6,
            \ t\in[0,1].
     \end{eqnarray*}
Thus, when $T_1=\min\{ \frac{1}{8C^3(1+M_q)^{2}},1\}$, we can easily
obtain (\ref{3VDD-full-E2.29}).
\end{proof}

Now, we apply the estimates of Lemma \ref{3VDD-full-L2.2} to close
the bounds in  Lemma \ref{3VDD-full-L2.1}.

\begin{prop}\label{3VDD-full-P2.2}
If $(\rho,u)$
  is a smooth solution of (\ref{3VDD-full-E1.1})--(\ref{3VDD-full-E1.2}) as
  in Proposition \ref{3VDD-full-P2.1}, and $\varepsilon$ is small enough, then we have
    \begin{equation}
      \sup_{t\in(0,T]\cap(0,T^*)}\int\left(\sigma
      |\nabla
      u|^2+\sigma^N\rho|\dot{u}|^2\right)dx
            +\int^{T\wedge T^*}_0\int\left(\sigma\rho|\dot{u}|^2+
            \sigma^N|\nabla\dot{u}|^2
            \right)dxdt
        \leq(C_0+C_f)^\theta.\label{3VDD-full-E2.22}
    \end{equation}
\end{prop}
\begin{proof}
Since we obtain this proposition in $\mathbb{R}^2$ in
\cite{Zhang2008} (Proposition 2.2), then we only prove this
proposition in $\mathbb{R}^3$ in this paper.

  From Proposition \ref{3VDD-full-P2.1} and Lemmas
  \ref{3VDD-full-L2.1}--\ref{3VDD-full-L2.2}, we have
        \begin{equation}
          \textrm{LHS of }(\ref{3VDD-full-E2.22})\leq C(C_0+C_f)
          +C\int^{T\wedge T^*}_0\int_{\mathbb{R}^3}\left(\sigma|\nabla u|^3
          +\sigma^3|u|^4+\sigma^3|\nabla u|^4
          \right)dxds.\label{3VDD-full-E2.25-0}
        \end{equation}
From  (\ref{3VDD-full-E2.16-1}), we get
    \begin{equation}
      \int^{T\wedge T^*}_0\int_{\mathbb{R}^3}\sigma^3|\nabla u|^4dxds\leq
      \int^{T\wedge T^*}_0\int_{\mathbb{R}^3}\sigma^3\left[
      |F|^4+|w|^4+|P-P(\tilde{\rho})|^4
      \right]dxds.\label{3VDD-full-E2.26}
    \end{equation}
From (\ref{3VDD-full-E2.1-2}), (\ref{3VDD-full-E2.1}),
(\ref{3VDD-full-E2.16-0}), (\ref{3VDD-full-E2.16}) and
(\ref{3VDD-full-E2.16-2})--(\ref{3VDD-full-E2.16-4}) we obtain
    \begin{eqnarray}
        && \int^{T\wedge T^*}_0\int_{\mathbb{R}^3}\sigma^3\left(|F|^4+|w|^4\right)dxds\nonumber\\
                &\leq&
      C\int^{T\wedge T^*}_0\sigma^3\left[\left(\int_{\mathbb{R}^3} |F|^2dx
      \right)^\frac{1}{2}\left(\int_{\mathbb{R}^3} |\nabla F|^2dx
      \right)^\frac{3}{2}+\left(\int_{\mathbb{R}^3} |w|^2dx
      \right)^\frac{1}{2}\left(\int_{\mathbb{R}^3} |\nabla w|^2dx
      \right)^\frac{3}{2}\right]ds\nonumber\\
            &\leq&C\!\!\sup_{t\in[0,T]\cap[0,T^*)}\!\!\left\{
            \int_{\mathbb{R}^3}\sigma(|\nabla u|^2+|\rho-\tilde{\rho}|^2)dx
            \int_{\mathbb{R}^3}\sigma^3\left(\rho|\dot{u}|^2+|f|^2
            \right)dx
            \right\}^\frac{1}{2}
            \int^{T\wedge T^*}_0\!\!\!\!\!\!\int_{\mathbb{R}^3}\sigma\left(\rho|\dot{u}|^2+|f|^2
            \right)dxds\nonumber\\
                &\leq&C(C_0+C_f)^{2\theta}+C(C_0+C_f)^2,\label{3VDD-full-E2.28}
    \end{eqnarray}
        \begin{eqnarray}
        \int^{T\wedge T^*}_0\int_{\mathbb{R}^3}\sigma^3|\rho-\tilde{\rho}|^4dxds&\leq&
        C\left(\int^{T\wedge T^*}_0\int_{\mathbb{R}^3}\sigma^3|F|^4dxds+C_0+C_f
        \right)\nonumber\\
            &\leq&C(C_0+C_f)^{2\theta}+C(C_0+C_f),
        \end{eqnarray}
    \begin{eqnarray}
        \int^{T\wedge T^*}_0\int_{\mathbb{R}^3}\sigma^3|u|^4dxds&\leq&
        C(C_0+C_f)^\frac{1}{3}\int^{T\wedge T^*}_0\sigma^3\left(\|\nabla u\|_{L^2}^3+\|\nabla
        u\|_{L^2}^4
        \right)ds\nonumber\\
            &\leq&CA_1^\frac{1}{2}(C_0+C_f)^\frac{4}{3}+CA_1(C_0+C_f)^\frac{4}{3}.\label{3VDD-full-E2.30}
        \end{eqnarray}
From (\ref{3VDD-full-E2.26})--(\ref{3VDD-full-E2.30}), we have
    \begin{equation}
      \int^{T\wedge T^*}_0\int_{\mathbb{R}^3}\sigma^3(|u|^4+|\nabla u|^4)dxds
      \leq C (C_0+C_f)^{2\theta}+C(C_0+C_f).\label{3VDD-full-E2.31}
    \end{equation}
    Similarly, we get
     \begin{eqnarray}
      &&\int^{T\wedge T^*}_{T_1\wedge T\wedge T^*}\int_{\mathbb{R}^3}\sigma|\nabla u|^3dxds\nonumber\\
      &\leq&
      \int^{T\wedge T^*}_{T_1\wedge T\wedge T^*}\int_{\mathbb{R}^3}\left(\sigma^2|\nabla u|^4+|\nabla u|^2
      \right)dxds\nonumber\\
      &\leq& C(T_1)
      \int^{T\wedge T^*}_{T_1\wedge T\wedge T^*}\int_{\mathbb{R}^3}\left(\sigma^3|\nabla u|^4+|\nabla u|^2
      \right)dxds\nonumber\\
        &\leq& C(M_q) (C_0+C_f)^{2\theta}+C(M_q)(C_0+C_f),\label{3VDD-full-E2.32}
    \end{eqnarray}
    \begin{eqnarray}
      &&\int^{T_1\wedge T\wedge T^*}_0\int_{\mathbb{R}^3}\sigma|\nabla u|^3dxds\nonumber\\
      &\leq& C(C_0+C_f)+
      \int^{T_1\wedge T\wedge T^*}_0\int_{\mathbb{R}^3}\sigma\left(|F|^3+|w|^3
      \right)dxds\nonumber\\
        &\leq& C(C_0+C_f)+
      \int^{T_1\wedge T\wedge T^*}_0\sigma\left(\int_{\mathbb{R}^3}(|\nabla u|^2+|\rho-\tilde{\rho}|^2)dx\right)^\frac{3}{4}
            \left(\int_{\mathbb{R}^3}\left(\rho|\dot{u}|^2+|f|^2
            \right)dx\right)^\frac{3}{4}ds\nonumber\\
      &\leq& C(M_q)(C_0+C_f)+C(M_q)(C_0+C_f)^\frac{3}{4}A_1^\frac{3}{4}+
      \left(\int^{T\wedge
      T^*}_0\sigma\|\sqrt{\rho}\dot{u}\|_{L^2}^2ds
      \right)^\frac{3}{4}\nonumber\\
        &&\times
      \left(\sup_{t\in[0,T_1\wedge T]\cap[0, T^*)}
      \sigma\|\nabla u\|_{L^2}^4\int^{T\wedge
      T^*}_0\|\nabla u\|_{L^2}^2ds+C\int^{T\wedge
      T^*}_0\|\rho-\tilde{\rho}\|_{L^2}^6dt
      \right)^{\frac{1}{4}}
      \nonumber\\
      &\leq& C(M_q)(C_0+C_f)+C(M_q)(C_0+C_f)^{\frac{3}{4}(1+\theta)}
      +C(M_q)(C_0+C_f)^{\frac{1}{4}+\theta}.\label{3VDD-full-E2.32-10}
    \end{eqnarray}
Then, from (\ref{3VDD-full-E2.1-2}), (\ref{3VDD-full-E2.25-0}) and
(\ref{3VDD-full-E2.31})--(\ref{3VDD-full-E2.32-10}), we obtain
     \begin{eqnarray}
         && \textrm{LHS of }(\ref{3VDD-full-E2.22})\nonumber\\
         &\leq& C(M_q)(C_0+C_f)^{1\wedge2\theta\wedge\frac{3}{4}(1+\theta)\wedge(\frac{1}{4}+\theta)}\nonumber\\
          &\leq&
          (C_0+C_f)^\theta ,
        \end{eqnarray}
when
    \begin{equation}
      C(M_q) \varepsilon^{(1-\theta)\wedge \theta\wedge(\frac{3}{4}-\frac{1}{4}\theta)\wedge\frac{1}{4}}
      \leq1.
    \end{equation}
\end{proof}

Then, we consider the H\"{o}lder continuity of $u$ in the
following lemma.
\begin{lem}
  Let $\alpha\in(0,1)$ when $N=2$, $\alpha\in(0,\frac{1}{2}]$ when
$N=3$. When $t\in(0,T]\cap(0,T^*)$, we have
    \begin{equation}
    <u(\cdot,t)>^\alpha\leq C\left(\|\rho\dot{u}\|_{L^2}^{\frac{N-2+2\alpha}{2}}
        (\|\nabla
        u\|_{L^2}^{\frac{4-N-2\alpha}{2}}+(C_0+C_f)^{\frac{4-N-2\alpha}{4}})
    +\|\nabla u\|_{L^2}+(C_0+C_f)^\frac{1-\alpha}{N}
         \right).\label{3VDD-full-E2.35}
    \end{equation}
\end{lem}
\begin{proof} Let  $p=\frac{N}{1-\alpha}$.
 From (\ref{3VDD-full-E2.16-1}),
  (\ref{3VDD-full-E2.16-3}) and Sobolev's
embedding theorem, we have
    \begin{eqnarray*}
    &&<u(\cdot,t)>^\alpha\\
        &\leq& C\|\nabla u\|_{L^p}\nonumber\\
    &\leq&C(\|F\|_{L^p}
          +\|w\|_{L^p}+\|P-P(\tilde{\rho})\|_{L^p}
          )\nonumber\\
    &\leq&C\left(\|\rho\dot{u}\|_{L^2}^{\frac{Np-2N}{2p}}
        (\|\nabla
        u\|_{L^2}^{\frac{2N-Np+2p}{2p}}+\|P-P(\tilde{\rho})\|_{L^2}^{\frac{2N-Np+2p}{2p}})
    +\|\nabla u\|_{L^2}+\|f\|_{L^2}+\|P-P(\tilde{\rho})\|_{L^2}
         \right)\nonumber\\
            &&+C\|\rho-\tilde{\rho}\|_{L^2}^\frac{2}{p}
            \|\rho-\tilde{\rho}\|_{L^\infty}^{1-\frac{2}{p}}\\
    &\leq&C\left(\|\rho\dot{u}\|_{L^2}^{\frac{Np-2N}{2p}}
        (\|\nabla
        u\|_{L^2}^{\frac{2N-Np+2p}{2p}}+(C_0+C_f)^{\frac{2N-Np+2p}{4p}})
    +\|\nabla u\|_{L^2}+(C_0+C_f)^\frac{1}{p}
         \right).
    \end{eqnarray*}
\end{proof}

\begin{prop}\label{3VDD-full-P2.5}
 If $u_0\in H^1$, $(\rho,u)$
  is a smooth solution of (\ref{3VDD-full-E1.1})--(\ref{3VDD-full-E1.2}) as
  in Proposition \ref{3VDD-full-P2.1}, then we have
    \begin{equation}
      \sup_{t\in[0,T]\cap[0,T^*)}\int
      |\nabla
      u|^2dx+\int^{T\wedge T^*}_0\int\rho|\dot{u}|^2dxdt
        \leq C(M_q).\label{3VDD-full-E2.59}
    \end{equation}
\end{prop}
\begin{proof}
  Using a similar argument as that in the proof of
  (\ref{3VDD-full-E2.4}), we have
    $$
      \sup_{t\in[0,T]\cap[0,T^*)}\int|\nabla
      u|^2dx+\int^{T\wedge T^*}_0\int\rho|\dot{u}|^2dxdt
            \leq C(C_0+C_f+M_q)+C\int^{T\wedge T^*}_0\int|\nabla u|^3dxds.
   $$
 From Lemma
\ref{3VDD-full-L2.3} and (\ref{3VDD-full-E2.32}), we can easily
obtain (\ref{3VDD-full-E2.59}).
\end{proof}

\begin{prop}\label{3VDD-full-P2.6}
 If $u_0\in H^1$, $(\rho,u)$
  is a smooth solution of (\ref{3VDD-full-E1.1})--(\ref{3VDD-full-E1.2}) as
  in Proposition \ref{3VDD-full-P2.1}, then we have
    \begin{equation}
      \sup_{t\in(0,T]\cap(0,T^*)}\sigma\int\rho|\dot{u}|^2dx
      +\int^{T\wedge T^*}_0\int\sigma|\nabla \dot{u}|^2dxdt
        \leq C(M_q).\label{3VDD-full-E2.62}
    \end{equation}
\end{prop}
\begin{proof}
Since we obtain this proposition in $\mathbb{R}^2$ in
\cite{Zhang2008} (Proposition 2.4), then we only prove this
proposition in $\mathbb{R}^3$ in this paper.

  Using a similar argument as that in the proof of
  (\ref{3VDD-full-E2.5}), from (\ref{3VDD-full-E2.59}), we have
    $$
      \sup_{t\in(0,T]\cap(0,T^*)}\sigma\int_{\mathbb{R}^3}\rho|\dot{u}|^2dx
      +\int^{T\wedge T^*}_0\int_{\mathbb{R}^3}\sigma|\nabla \dot{u}|^2dxdt
            \leq C(M_q)+C\int^{T\wedge T^*}_0\int_{\mathbb{R}^3}\sigma\left(|u|^4+|\nabla u|^4\right)dxds.
   $$
Without loss of generality, assume that $T>1$. From
 (\ref{3VDD-full-E2.31}), we get
        $$
   \sup_{t\in(0,T]\cap(0,T^*)}\sigma\int_{\mathbb{R}^3}\rho|\dot{u}|^2dx
      +\int^{T\wedge T^*}_0\int_{\mathbb{R}^3}\sigma|\nabla \dot{u}|^2dxdt
            \leq C(M_q)+C\int^{1\wedge T^*}_0\int_{\mathbb{R}^3}\sigma\left(|u|^4+|\nabla u|^4\right)dxds.
   $$
  From (\ref{3VDD-full-E2.1}),
(\ref{3VDD-full-E2.16-0})--(\ref{3VDD-full-E2.16-1})
 and (\ref{3VDD-full-E2.59}), we have
    $$
\sup_{t\in(0,T]\cap(0,T^*)}\sigma\int_{\mathbb{R}^3}\rho|\dot{u}|^2dx
      +\int^{T\wedge T^*}_0\int_{\mathbb{R}^3}\sigma|\nabla \dot{u}|^2dxdt
            \leq C(M_q)+C\int^{1\wedge T^*}_0\int_{\mathbb{R}^3}\sigma(|F|^4+|w|^4)dxds.
    $$
  From (\ref{3VDD-full-E2.1}),
(\ref{3VDD-full-E2.16-2}), (\ref{3VDD-full-E2.16}) and
(\ref{3VDD-full-E2.59}), we obtain
    \begin{eqnarray}
        && \int^{1\wedge T^*}_0\int_{\mathbb{R}^3}\sigma\left(|F|^4+|w|^4\right)dxds\nonumber\\
                &\leq&
      C\int^{1\wedge T^*}_0\sigma\left(\int_{\mathbb{R}^3} |F|^2dx
      \right)^{\frac{1}{2}}\left(\int_{\mathbb{R}^3} |\nabla F|^2dx
      \right)^{\frac{3}{2}}ds+
      C\int^{1\wedge T^*}_0\sigma\left(\int_{\mathbb{R}^3} |w|^2dx
      \right)^{\frac{1}{2}}\left(\int_{\mathbb{R}^3} |\nabla w|^2dx
      \right)^{\frac{3}{2}}ds\nonumber\\
            &\leq&C\int^{1\wedge T^*}_0\sigma\left(\int_{\mathbb{R}^3}(|\nabla
            u|^2+|\rho-\tilde{\rho}|^2)dx\right)^{\frac{1}{2}}
            \left(\int_{\mathbb{R}^3}\left(\rho|\dot{u}|^2+|f|^2
            \right)dx\right)^{\frac{3}{2}}ds\nonumber\\
                &\leq&C(M_q)+
                C(M_q)\sup_{t\in(0,T]\cap(0,T^*)}
                \sigma^\frac{1}{2}\|\sqrt{\rho}\dot{u}\|_{L^2}.
    \end{eqnarray}
Using Young's inequality, we can finish the proof of this
proposition.
\end{proof}

\begin{lem} For any $p\in[2,\infty)$ when $N=2$, $p\in[2,6]$ when $N=3$, we have
    \begin{equation}
     \|\dot{u}\|_{L^p}\leq C_p\|\sqrt{\rho}\dot{u}\|_{L^2}^{\frac{2N-Np+2p}{2p}}
     \|\nabla\dot{u}\|_{L^2}^{\frac{Np-2N}{2p}}+C_p\|\nabla\dot{u}\|_{L^2}.\label{3VDD-full-E2.64-0}
    \end{equation}
\end{lem}
\begin{proof}
  Since
    $$
    \tilde{\rho}\int|\dot{u}|^2dx\leq \int\rho |\dot{u}|^2dx+
    \left(\int|\rho-\tilde{\rho}|^2
    dx    \right)^\frac{1}{2}
    \left(\int|\dot{u}|^4
    dx    \right)^\frac{1}{2},
    $$
applying (\ref{3VDD-full-E2.16}), we get
         $$
    \|\dot{u}\|_{L^2}^2\leq
    C\|\sqrt{\rho}\dot{u}\|_{L^2}^2+C\|\nabla \dot{u}\|_{L^2}^2.
    $$
From (\ref{3VDD-full-E2.16}), we can immediately obtain
(\ref{3VDD-full-E2.64-0}).
\end{proof}

\begin{lem}
  For any $q\in (0,2)$ when $N=2$, $q\in (1,\frac{4}{3})$ when $N=3$, we have
    \begin{equation}
      \int^{T\wedge T^*}_0\int\sigma^{p_1-1}\rho|\dot{u}|^{2+q}dxds
      \leq C(M_q).\label{3VDD-full-E2.65}
    \end{equation}
\end{lem}
\begin{proof}
Since we obtain this lemma in $\mathbb{R}^2$ in \cite{Zhang2008}
(lemma 2.5), then we only prove this lemma in $\mathbb{R}^3$ in this
paper.

 Using H\"{o}lder's inequality, (\ref{3VDD-full-E2.59}),
(\ref{3VDD-full-E2.62}) and (\ref{3VDD-full-E2.64-0}) with $p=6$, we
have
    \begin{eqnarray*}
    &&\int^{T\wedge T^*}_0\int_{\mathbb{R}^3}\sigma^{\frac{5q}{4}}\rho|\dot{u}|^{2+q}dxds\\
        &\leq&C\int^{T\wedge T^*}_0\sigma^{\frac{5q}{4}}\|\sqrt{\rho}\dot{u}\|_{L^2}^\frac{4-q}{2}
        \|\dot{u}\|_{L^6}^{\frac{3q}{2}}ds\\
         &\leq&C\int^{T\wedge T^*}_0\sigma^{\frac{5q}{4}}\|\sqrt{\rho}\dot{u}\|_{L^2}^\frac{4-q}{2}
        \|\nabla\dot{u}\|_{L^2}^{\frac{3q}{2}}ds\\
            &\leq& C\left(\int^{T\wedge T^*}_0\sigma\|\nabla \dot{u}\|_{L^2}^2dt
            \right)^\frac{3q}{4}
            \left(\int^{T\wedge T^*}_0\|\sqrt{\rho} \dot{u}\|_{L^2}^2dt
            \right)^\frac{4-3q}{4}
            \left(\sup_{t\in[0,T]}\sigma\|\sqrt{\rho} \dot{u}\|_{L^2}^2
            \right)^\frac{q}{2}\\
                        &\leq& C(M_q).
    \end{eqnarray*}
\end{proof}

\begin{prop}\label{3VDD-full-P2.7}
 If $u_0\in H^1$, $(\rho,u)$
  is a smooth solution of (\ref{3VDD-full-E1.1})--(\ref{3VDD-full-E1.2}) as
  in Proposition \ref{3VDD-full-P2.6} and
    \begin{equation}
      q^2<\frac{4\mu}{\lambda(\rho)+\mu},
     \ \forall\ \rho\in[0,\bar{\rho}],\label{3VDD-full-E2.64}
    \end{equation}
  then we have
    \begin{equation}
      \sup_{t\in(0,T]\cap(0,T^*)}\sigma^{p_1}\int\rho|\dot{u}|^{2+q}dx
      +\int^{T\wedge T^*}_0\int\sigma^{p_1}|\dot{u}|^q|\nabla \dot{u}|^2dxdt
        \leq C(M_q).\label{3VDD-full-E2.66}
    \end{equation}
\end{prop}
\begin{proof}
In \cite{Zhang2008} (Proposition 2.5), we obtain this proposition in
$\mathbb{R}^2$. Using the similar argument as that in
\cite{Zhang2008} (Proposition 2.5), we can easily obtain
 this proposition in
$\mathbb{R}^3$ and omit the details.
\end{proof}

\begin{prop}\label{3VDD-full-P2.8}
 If $f\in L^\infty_tL^{2+q}_x$, $(\rho,u)$
  is a smooth solution of (\ref{3VDD-full-E1.1})--(\ref{3VDD-full-E1.2}) as
  in Proposition \ref{3VDD-full-P2.7}, then we have
    \begin{equation}
      \|F\|_{L^\infty}+\|w\|_{L^\infty}\leq C(\|\nabla u\|_{L^2}+\|\rho-\tilde{\rho}\|_{L^2})^{\frac{2(2+q-N)}{4+2q+Nq}}
        (\|\rho\dot{u}\|_{L^{2+q}}+\|f\|_{L^{2+q}})^{\frac{2N+Nq}{4+2q+Nq}}\label{3VDD-full-E2.72-1}
    \end{equation}
  and
    \begin{equation}
     \int^{T\wedge T^*}_0(\|F\|_{L^\infty}+\|w\|_{L^\infty})ds
        \leq C(M_q)(C_0+C_f)^{\frac{\theta(2+q-N)}{4+2q+Nq}}(1+T).\label{3VDD-full-E2.72}
    \end{equation}
\end{prop}
\begin{proof}
  From (\ref{3VDD-full-E2.1}), (\ref{3VDD-full-E2.16-2}),
  (\ref{3VDD-full-E2.22}),
   (\ref{3VDD-full-E2.66}) and  the Galiardo-Nirenberg inequality, we have
    \begin{eqnarray*}
        &&    \|F\|_{L^\infty}\\
        &\leq& C\|F\|_{L^2}^{\frac{2(2+q-N)}{4+2q+Nq}}
        \|\nabla F\|_{L^{2+q}}^{\frac{2N+Nq}{4+2q+Nq}}\\
        &\leq& C(\|\nabla u\|_{L^2}+\|\rho-\tilde{\rho}\|_{L^2})^{\frac{2(2+q-N)}{4+2q+Nq}}
        (\|\rho\dot{u}\|_{L^{2+q}}+\|f\|_{L^{2+q}})^{\frac{2N+Nq}{4+2q+Nq}}.
    \end{eqnarray*}
  and
    \begin{eqnarray*}
        &&   \int^{T\wedge T^*}_0 \|F\|_{L^\infty}ds\\
        &\leq& C(M_q)\int^{T\wedge T^*}_0(\sigma^{-\frac{1}{2}}(C_0+C_f)^\frac{\theta}{2})^{\frac{2(2+q-N)}{4+2q+Nq}}
        (\sigma^{-\frac{p_1}{2+q}})^{\frac{2N+Nq}{4+2q+Nq}}ds\\
        &\leq& \leq C(M_q)(C_0+C_f)^{\frac{\theta(2+q-N)}{4+2q+Nq}}(1+T).
    \end{eqnarray*}
Similarly, we can obtain the same estimates for $w$.
\end{proof}

 Then, we derive a priori pointwise bounds for the density
$\rho$.

\begin{prop}\label{3VDD-full-P2.3}
   Given numbers $0<\underline{\rho}_1<\tilde{\rho}
  <\bar{\rho}_1<\bar{\rho}_2<\bar{\rho}$, there is an
  $\varepsilon>0$ such that, if $(\rho,u)$
  is a smooth solution of (\ref{3VDD-full-E1.1})--(\ref{3VDD-full-E1.2}) with
  $C_0+C_f\leq\varepsilon$ and
  $\underline{\rho}_1\leq\rho_0\leq\bar{\rho}_1$,
   then
        \begin{equation}
          \frac{1}{2}\underline{\rho}_1\leq\rho\leq\bar{\rho}_2,
          \ (x,t)\in\mathbb{R}^N\times\{[0,T]\cap[0,T^*)\},\label{3VDD-full-E2.37-0}
        \end{equation}
   for any $T>0$.
   Furthermore, Claim 1 and
  the estimates in Propositions \ref{3VDD-full-P2.1}--\ref{3VDD-full-P2.8}  hold for any $T>0$.
\end{prop}
\begin{proof}
At first, we prove that if (\ref{3VDD-full-E2.1-1}) and
(\ref{3VDD-full-E2.1-2}) hold, then estimate
(\ref{3VDD-full-E2.37-0}) holds.

  We fix a curve $x(t)$ satisfying $\dot{x}= u(x(t),t)$ and $x(0)=x$. From
  (\ref{3VDD-full-E1.18}), we have
    \begin{equation}
 \frac{d}{dt}\Lambda(\rho(x(t),t))+P(\rho(x(t),t))-P(\tilde{\rho})=-F
 (x(t),t),\label{3VDD-full-E2.36}
    \end{equation}
where $\Lambda$ satisfies that $\Lambda(\tilde{\rho})=0$ and
$\Lambda'(\rho)=\frac{2\mu+\lambda(\rho)}{\rho}$.

\textbf{(I)} For the small time, we estimate the pointwise bounds of
the density as follows. From (\ref{3VDD-full-E2.1-1}) and
(\ref{3VDD-full-E2.72}), we have, for all $t\in[0,1]$,
    $$
    \left|\Lambda(\rho(x(t),t))-\Lambda(\rho_0(x))
    \right|
        \leq C(M_q)(C_0+C_f)^{\frac{\theta(2+q-N)}{4+2q+Nq}}+Ct.
    $$
When
    \begin{equation}
    2C(M_q)\varepsilon^{\frac{\theta(2+q-N)}{4+2q+Nq}}\leq
        \Lambda(\bar{\rho}_1+\frac{1}{3}(\bar{\rho}_2-\bar{\rho}_1))-\Lambda(\bar{\rho}_1),
    \end{equation}
we get
    $$
      \Lambda(\rho(x(t),t))\leq \Lambda(\bar{\rho}_1+\frac{1}{3}(\bar{\rho}_2-\bar{\rho}_1)),
      \ \ t\in [0,\tau],
    $$
and
    \begin{equation}
          \rho\leq\bar{\rho}_1+\frac{1}{3}(\bar{\rho}_2-\bar{\rho}_1),
          \ (x,t)\in\mathbb{R}^N\times[0,\tau],
        \end{equation}
where
$\tau=\min\{1,\frac{1}{2C}[\Lambda(\bar{\rho}_1+\frac{1}{3}(\bar{\rho}_2-\bar{\rho}_1))
-\Lambda(\bar{\rho}_1)]\}$. Similarly, since
    $$
\Lambda(\underline{\rho}_1)-\Lambda(\frac{5}{6}\underline{\rho}_1)\geq\int^{\underline{\rho}_1}_{\frac{5}{6}\underline{\rho}_1}
\frac{2\mu}{s}ds=2\mu\ln\frac{6}{5},
    $$
then, if
     \begin{equation}
    2C(M_q)\varepsilon^{\frac{\theta(2+q-N)}{4+2q+Nq}}
    \leq 2\mu\ln\frac{6}{5}\leq \Lambda(\underline{\rho}_1)-\Lambda(
        \frac{5}{6}\underline{\rho}_1),
    \end{equation}
we get
       \begin{equation}
          \rho\geq\frac{5}{6}\underline{\rho}_1,
          \ (x,t)\in\mathbb{R}^N\times[0,\tau_1],
        \end{equation}
where $\tau_1=\min\{\tau, \frac{\mu}{C}\ln\frac{6}{5}\}$.

\textbf{(II)} For the large time $t\geq\tau_1$, we estimate the
pointwise bounds of density as follows. From (\ref{3VDD-full-E2.1}),
(\ref{3VDD-full-E2.22}), (\ref{3VDD-full-E2.66}),
 (\ref{3VDD-full-E2.72-1}) and (\ref{3VDD-full-E2.36}), we have
        \begin{equation}
      \frac{d \Lambda(\rho(x(t),t))}{dt}+P(\rho(x(t),t))-P(\tilde{\rho})=O_5(t),\label{3VDD-full-E2.55}
    \end{equation}
where
    $$
    |O_5(t)|\leq  C(\tau_1,M_q)(C_0+C_f)^{\frac{\theta(2+q-N)}{4+2q+Nq}},
    \ t\geq\tau_1.
    $$
Now, we apply a standard maximum principle argument to estimate
the upper bounds of density. Let
    $$
    t_0=\sup\{t\in(\tau,T]\cap (\tau, T^*)|\Lambda(\rho(x(s),s))\leq\Lambda(
    \bar{\rho}_2),\
    \textrm{ for all }s\in[0,t]
    \}.
    $$
If $t_0<T$ and $t_0<T^*$, we have
    $$
    \Lambda(\rho(x(t_0),t_0))=\Lambda(\bar{\rho}_2),
    $$
     $$
   \left. \frac{d\Lambda(\rho(x(t),t))}{dt}\right|_{t=t_0}\geq0,
    $$
and
    $$
    \rho(x(t_0),t_0)=\bar{\rho}_2.
    $$
From (\ref{3VDD-full-E2.55}), we have
    $$
    O_5(t_0)\geq P(\bar{\rho}_2)-P(\tilde{\rho}).
    $$
On the other hand, when
    \begin{equation}
    C(\tau_1,M_q)        \varepsilon^{\frac{\theta(2+q-N)}{4+2q+Nq}}<
        P(\bar{\rho}_2)-P(\tilde{\rho}),
    \end{equation}
we have
     $$
    O_5(t_0)< P(\bar{\rho}_2)-P(\tilde{\rho}).
    $$
It is a contradiction. Thus, we have
        \begin{equation}
          \rho\leq\bar{\rho}_2,
          \ (x,t)\in\mathbb{R}^N\times\{[0,T]\cap [0,T^*)\}.
        \end{equation}
Similarly,  let
    $$
    t_1=\sup\{t\in(\tau,T]\cap (\tau,
    T^*)|\Lambda(\rho(x(s),s))\geq\Lambda(\frac{1}{2}
    \underline{\rho}_1),\
    \textrm{ for all }s\in[0,t]
    \}.
    $$
If $t_1<T$ and $t_1<T^*$, we have
    $$
    \Lambda(\rho(x(t_1),t_1))=\Lambda(\frac{1}{2}\underline{\rho}_1),
    $$
     $$
   \left. \frac{d\Lambda(\rho(x(t),t))}{dt}\right|_{t=t_1}\leq0,
    $$
and
    $$
    \rho(x(t_1),t_1)=\frac{1}{2}\underline{\rho}_1.
    $$
From (\ref{3VDD-full-E2.55}), we have
    $$
    O_5(t_1)\leq P(\frac{1}{2}\underline{\rho}_1)-P(\tilde{\rho})
    \leq \max_{s\in[0,\frac{1}{2}\tilde{\rho}]}P(s)-P(\tilde{\rho}).
    $$
On the other hand, when
    \begin{equation*}
    C(\tau_1,M_q)
    \varepsilon^{\frac{\theta(2+q-N)}{4+2q+Nq}}<P(\tilde{\rho})-
        \max_{s\in[0,\frac{1}{2}\tilde{\rho}]}P(s),
    \end{equation*}
we have
     $$
    O_5(t_1)> \max_{s\in[0,\frac{1}{2}\tilde{\rho}]}P(s)-P(\tilde{\rho}).
    $$
It is a contradiction. Thus, we have
        \begin{equation*}
          \rho\geq\frac{1}{2}\underline{\rho}_1,
          \ (x,t)\in\mathbb{R}^N\times\{[0,T]\cap [0,T^*)\}.
        \end{equation*}
Using the classical continuation method, (\ref{3VDD-full-E2.22}) and
(\ref{3VDD-full-E2.37-0}), we can finish the proof of this
proposition.
\end{proof}

 From now on, the constant $K$ ($K(T)$) will depend on
$\underline{\rho}_1$ (and $T$).

\begin{lem}
 For any $T>0$, we have
    \begin{equation}
     \sup_{t\in[0,T]\cap[0,T^*)} \|\nabla\rho(\cdot,t)\|_{L^{2+q}}+\int^{T\wedge T^*}_0\left(\|\nabla u\|_{L^\infty}
     +\|\Delta u\|_{L^{2+q}}
     \right)dt\leq K(T), \label{3VDD-full-E2.62-0}
    \end{equation}
    \begin{equation}
     \sup_{t\in[0,T]\cap[0,T^*)}
      \|\nabla\rho(\cdot,t)\|_{L^{2}}\leq K(T). \label{3VDD-full-E2.62-2}
    \end{equation}
\end{lem}
\begin{proof}
  From
  (\ref{3VDD-full-E1.18}), we have
    \begin{equation}
 \partial_t\Lambda(\rho(x,t))+u\cdot\nabla \Lambda+P(\rho(x,t))-P(\tilde{\rho})=-F
 (x,t),\label{3VDD-full-E2.63}
    \end{equation}
where $\Lambda$ satisfies that $\Lambda(\tilde{\rho})=0$ and
$\Lambda'(\rho)=\frac{2\mu+\lambda(\rho)}{\rho}$. By the simple
computation, we have
    $$
    \|\nabla \Lambda(t)\|_{L^{2+q}}^{2+q}\leq
    \|\nabla \Lambda(0)\|_{L^{2+q}}^{2+q}+K\int^t_0\left(\|\nabla
    \Lambda\|_{L^{2+q}}^{2+q}+\|\nabla F\|_{L^{2+q}}\|\nabla
    \Lambda\|_{L^{2+q}}^{1+q}+\|\nabla u\|_{L^\infty}\|\nabla
    \Lambda\|_{L^{2+q}}^{2+q}\right)ds.
    $$
Using the Fourier analysis methods, one can obtain the following
estimate.
    \begin{equation}
      \|\nabla u\|_{L^\infty}\leq C\|u\|_{L^2}+C(\|\nabla
      u\|_{\dot{B}^0_{\infty,\infty}}+1)\log\left(
      e+\|\Delta u\|_{L^{2+q}}
      \right).
    \end{equation}
(For the convenience of reader's reading, we also give the proof in
Appendix (\ref{2VDD-full-E3.1}).) From (\ref{3VDD-full-E2.24}), we
have
    $$
    \|\Delta u\|_{L^{2+q}}\leq K(\|\nabla F\|_{L^{2+q}}+\|F\|_{L^\infty}\|\nabla \Lambda\|_{L^{2+q}}
    +\|\nabla \Lambda\|_{L^{2+q}}+\|\nabla w\|_{L^{2+q}}),
    $$
and
    $$
    \|\nabla u\|_{\dot{B}^0_{\infty,\infty}}\leq
    C(\|F\|_{L^\infty}+\|\rho-\tilde{\rho}\|_{L^\infty}+\|w\|_{L^\infty}).
    $$
Thus, we have
    $$
    \|\nabla \Lambda(t)\|_{L^{2+q}}^{2+q}\leq \|\nabla \Lambda(0)
    \|_{L^{2+q}}^{2+q}+K\int^t_0\|\nabla
    F\|_{L^{2+q}}\|\nabla \Lambda\|_{L^{2+q}}^{1+q}ds
        +K\int^t_0\mathcal{A}\log\left(\|\nabla \Lambda\|_{L^{2+q}}+e
    \right)\|\nabla \Lambda\|_{L^{2+q}}^{2+q}ds,
    $$
and
     \begin{eqnarray*}
   \sup_{t\in[0,T]\cap[0,T^*)} \|\nabla \Lambda(t)\|_{L^{2+q}}&\leq&
   \left(K+\|\nabla \Lambda(0)\|_{L^{2+q}}+K\int^T_0\|\nabla
    F\|_{L^{2+q}}ds\right)^{\exp\{
        K\int^T_0\mathcal{A}(s)ds\}},
    \end{eqnarray*}
where
$\mathcal{A}=\left(\|F\|_{L^\infty}+\|w\|_{L^\infty}+1\right)\log\left(\|F\|_{L^\infty}+\|w\|_{L^\infty}+\|\nabla
F\|_{L^{2+q}}+\|\nabla
    w\|_{L^{2+q}}+e\right)$. From (\ref{3VDD-full-E2.1}), (\ref{3VDD-full-E2.16-2}),
    (\ref{3VDD-full-E2.59}), (\ref{3VDD-full-E2.62}),
    (\ref{3VDD-full-E2.66}), (\ref{3VDD-full-E2.72-1}) and (\ref{3VDD-full-E2.37-0}), we can immediately
    obtain
    (\ref{3VDD-full-E2.62-0}). Similarly, we can obtain
    (\ref{3VDD-full-E2.62-2}).
\end{proof}

\begin{lem}\label{3VDD-full-P2.9}
 If $\rho_0-\tilde{\rho}\in H^1$ and $u_0\in H^2$,  then for any $T>0$, we have
    \begin{equation}
      \sup_{t\in[0,T]\cap[0,T^*)}\int\rho|\dot{u}|^2dx
      +\int^{T\wedge T^*}_0\int\left(|\nabla
      \dot{u}|^2+|\frac{D}{Dt}\mathrm{div}u|^2
      \right)dxdt
        \leq K,\label{3VDD-full-E2.65-10}
    \end{equation}
    \begin{equation}
      \sup_{t\in[0,T]\cap[0,T^*)}\left(\|u\|_{L^\infty}+\|u\|_{H^2}+\|\nabla F\|_{L^2}+\|\nabla w\|_{L^2}
      \right)
        \leq K(T).\label{3VDD-full-E2.65-11}
    \end{equation}
\end{lem}
\begin{proof}
  Using the similar argument as that in the proof of Proposition
\ref{3VDD-full-P2.6}, we can obtain (\ref{3VDD-full-E2.65-10}). From
(\ref{3VDD-full-E2.16-2}), we get
    \begin{equation}
      \sup_{t\in[0,T]\cap[0,T^*)}\left(\|\nabla F\|_{L^2}+\|\nabla w\|_{L^2}
      \right)
        \leq K.\label{3VDD-full-E2.67-0}
    \end{equation}
From (\ref{3VDD-full-E2.1}), (\ref{3VDD-full-E2.24}),
(\ref{3VDD-full-E2.59}), (\ref{3VDD-full-E2.37-0}),
(\ref{3VDD-full-E2.62-0}), (\ref{3VDD-full-E2.62-2}),
(\ref{3VDD-full-E2.67-0}), we have
    \begin{eqnarray*}
    \| u(\cdot,t)\|_{H^{2}}&\leq& K(\| u\|_{L^{2}}+\|\nabla F\|_{L^{2}}+\|F\nabla \rho\|_{L^{2}}
    +\|\nabla \rho\|_{L^{2}}+\|\nabla w\|_{L^{2}})\\
    &\leq& K(\| u\|_{L^{2}}+\|\nabla F\|_{L^{2}}+\|F\|_{L^2}^{\frac{2+q-N}{2+q}}
    \|\nabla F\|_{L^2}^{\frac{N}{2+q}}\|\nabla \rho\|_{L^{2+q}}
    +\|\nabla \rho\|_{L^{2}}+\|\nabla w\|_{L^{2}})\\
        &\leq& K(T),\ \ \ t\in[0,T]\cap [0,T^*).
    \end{eqnarray*}
Then, using Sobolev's embedding theorem, we can finish this proof.
\end{proof}

\begin{lem}\label{3VDD-full-L2.7}
 For any $T>0$, we have
    \begin{equation}
     \sup_{t\in[0,T]\cap[0,T^*)} \|\rho(\cdot,t)-\tilde{\rho}\|_{H^2}\leq K(T). \label{3VDD-full-E2.65-0}
    \end{equation}
\end{lem}
\begin{proof}
From
  (\ref{3VDD-full-E2.63}) and the simple
computation, we have
    \begin{equation}
    \|\Lambda(t)\|_{H^2}^{2}\leq \|\Lambda(0)\|_{H^2}^{2}+K\int^t_0\left(\|F\|_{H^{2}}\|
    \Lambda\|_{H^{2}}+(1+\|\nabla u\|_{L^\infty}+\| u\|_{H^2})\|
    \Lambda\|_{H^{2}}^{2}\right)ds.\label{3VDD-full-E2.70}
    \end{equation}
From  (\ref{3VDD-full-E1.20}),  (\ref{3VDD-full-E2.59}),
(\ref{3VDD-full-E2.37-0}), (\ref{3VDD-full-E2.62-2}),
(\ref{3VDD-full-E2.65-10}), the Galiardo-Nirenberg inequality and
Sobolev's embedding theorem, we get
    \begin{eqnarray}
    \|F\|_{H^2}&\leq& K\left(\|F\|_{L^2}+\|\nabla \rho \dot{u}\|_{L^2}
    +\|\rho\nabla \dot{u}\|_{L^2}+\|\nabla f\|_{L^2}+\|f\cdot\nabla \rho\|_{L^2}
    \right)\nonumber\\
        &\leq&K(T)\left(1+\|\nabla \rho\|_{L^3}\| \dot{u}\|_{L^6}
    +\|\nabla \dot{u}\|_{L^2}
    \right)\nonumber\\
            &\leq&K(T)\left(1+\|\nabla\rho\|_{L^2}^{\frac{6-N}{6}}\|\nabla^2 \rho\|_{L^2}^{\frac{N}{6}}
         \| \dot{u}\|_{L^2}^\frac{3-N}{3} \|\nabla
         \dot{u}\|_{L^2}^\frac{N}{3}
    +\|\nabla \dot{u}\|_{L^2}
    \right)\nonumber\\
        &\leq&K(T)\left(1+\| \Lambda\|_{H^2}(1+\|\nabla
        \dot{u}\|_{L^2})
        +\|\nabla
        \dot{u}\|_{L^2}\right).\label{3VDD-full-E2.71-0}
    \end{eqnarray}
Thus, from (\ref{3VDD-full-E2.62-0}), (\ref{3VDD-full-E2.65-10}),
(\ref{3VDD-full-E2.70})--(\ref{3VDD-full-E2.71-0}) and Gronwall's
inequality, we have
    $$
    \|\Lambda(t)\|_{H^2}^{2}\leq K(T)+K(T)\int^t_0(1+\|\nabla u\|_{L^\infty}+\|\nabla
        \dot{u}\|_{L^2})\|
    \Lambda\|_{H^{2}}^{2}ds
    $$
and
    $$
    \|\Lambda(t)\|_{H^2}^{2}\leq K(T).
    $$
Using (\ref{3VDD-full-E2.37-0}) and (\ref{3VDD-full-E2.62-2}), we
can immediately obtain (\ref{3VDD-full-E2.65-0}).
\end{proof}

\begin{lem}\label{3VDD-full-L2.8}
For any $T>0$, we have
    \begin{equation}
        \int^{T\wedge T^*}_0\| u\|_{H^{3}}^2dt\leq
        K(T).\label{3VDD-full-E2.75}
    \end{equation}
\end{lem}
\begin{proof}
  From  (\ref{3VDD-full-E2.24-10}), (\ref{3VDD-full-E2.65-10}), (\ref{3VDD-full-E2.65-0}) and
  (\ref{3VDD-full-E2.71-0}), we have
     \begin{equation}
      \int^{T\wedge T^*}_0\left(\| F\|_{H^2}+\| w\|_{H^2}
      \right)^2dt
        \leq K(T).\label{3VDD-full-E2.76}
    \end{equation}
From (\ref{3VDD-full-E2.1}), (\ref{3VDD-full-E2.24}),
(\ref{3VDD-full-E2.65-0}) and (\ref{3VDD-full-E2.76}), we have
    $$
    \int^{T\wedge T^*}_0\| u\|_{H^{3}}^2dt\leq K\int^{T\wedge T^*}_0
    (\| u\|_{L^{2}}+(1+\|F\|_{H^{2}})(1+\|\rho-\tilde{\rho}\|_{H^{2}})
 +\|w\|_{H^{2}})^2dt
        \leq K(T).
    $$
\end{proof}

\begin{prop}
 For any $T>0$, we have
    \begin{equation}
      \sup_{t\in[0,T]\cap[0,T^*)}\int|\nabla\dot{u}|^2dx
      +\int^{T\wedge T^*}_0\int|\nabla^2 \dot{u}|^2dxdt
            \leq K(T),\label{3VDD-full-E2.72-5}
    \end{equation}
        \begin{equation}
    \sup_{t\in[0,T]\cap[0,T^*)}\left(\|(\rho-\tilde{\rho},u)\|_{H^3}+
    \|(\rho_t,u_t)\|_{H^2}\right)+\int^{T\wedge T^*}_0\|u\|_{H^4}^2ds
            \leq K(T).\label{3VDD-full-E2.72-50}
        \end{equation}
\end{prop}
\begin{proof}
Taking the operator $\nabla\partial_t+\nabla\mathrm{div}( u\cdot)$
in (\ref{3VDD-full-E1.1})$_2$, multiplying
 by $\nabla\dot{u}$ and integrating, we obtain
    \begin{eqnarray}
      &&\frac{1}{2}\int\rho|\nabla\dot{u}|^2dx
      \nonumber\\
            &=&\frac{1}{2}\int\rho_0|\nabla\dot{u}_0|^2dx+\int^t_0
            \int\{-\nabla \rho\partial_t\dot{u}\nabla
            \dot{u}-\nabla(\rho u_j)\partial_j \dot{u}\nabla
            \dot{u}
            \nonumber\\
                &&-\nabla\dot{u}^j\nabla[
            \partial_j P_t+\mathrm{div}(\partial_j Pu)]
            -\mu\Delta\dot{u}^j[\Delta u^j_t+\mathrm{div}(u\Delta
            u^j)]\nonumber\\
                    &&-\Delta\dot{u}^j[\partial_j\partial_t((\lambda+\mu)\mathrm{div}u)
                +\mathrm{div}(u\partial_j((\lambda+\mu)\mathrm{div}u))]
                -\Delta\dot{u}^j[(\rho f^j)_t+\mathrm{div}(u\rho f^j)]
            \}dxds\nonumber\\
                &:=&\sum^7_{i=1}J_i.\label{3VDD-full-E2.11-11}
    \end{eqnarray}
Since $\rho_0-\tilde{\rho}\in H^2$, $\rho_0\in
[\underline{\rho}_1,\bar{\rho}_1]$ and $u_0\in H^3$, we get
    \begin{equation}
      J_1=\frac{1}{2}\int\rho_0|\nabla\dot{u}_0|^2dx\leq K.
    \end{equation}
Using (\ref{3VDD-full-E2.37-0}),
(\ref{3VDD-full-E2.65-10})--(\ref{3VDD-full-E2.65-11}),
(\ref{3VDD-full-E2.65-0}), (\ref{3VDD-full-E2.75}), the integration
by parts and H\"{o}lder's inequality, we have
    \begin{eqnarray}
      J_2&=&-\int^t_0\int \nabla \rho\partial_t\dot{u}\nabla
            \dot{u}dxds\nonumber\\
            &=&-\int^t_0\int \nabla \rho\partial_t
            \left(\frac{\mu\Delta u+\nabla((\mu+\lambda)\mathrm{div}u)
            -\nabla P-\rho f}{\rho}\right)\nabla
            \dot{u}dxds\nonumber\\
            &\leq& K(T)\int^t_0\int|\nabla \rho||\nabla \dot{u}|\big(
            |\nabla^2 u||\nabla u|
            +|\nabla u||\nabla \rho|^2+|\nabla u|^2|\nabla \rho|
            +|\nabla \rho|^2+|\nabla\rho||\nabla^2u|+|\nabla^2u|\nonumber\\
                &&+|\nabla^2\rho|+|\nabla u||\nabla \rho|
            +|\nabla^3u|+|\nabla^2 \rho||\nabla u|+|\nabla \rho||\frac{D}{Dt}\mathrm{div}u|
            +|\Delta \dot{u}|+|\nabla\frac{D}{Dt}\mathrm{div}u|+|f_t|
            \big)dxds\nonumber\\
                &\leq&K(T)\int^t_0\left\{\|\nabla \dot{u}\|_{L^2}^\frac{1}{2}\|\nabla^2 \dot{u}\|_{L^2}^\frac{1}{2}
               \big(\|\nabla^2
                \dot{u}\|_{L^2}+\|\nabla\frac{D}{Dt}\mathrm{div}u\|_{L^2}+\|f_t\|_{L^2}\big)
                \|\rho-\tilde{\rho}\|_{H^2}\right.\nonumber\\
                    &&+\|\nabla \dot{u}\|_{L^2}^\frac{3-N}{3}\|\nabla^2 \dot{u}\|_{L^2}^\frac{N}{3}
                    \left.\left[
                    \left(\|u\|_{H^2}^2+1\right)\left(\|\nabla\rho\|_{H^1}^3+1
                \right)+\|\nabla\rho\|_{H^1}\|u\|_{H^3}
                +\left\|\frac{D}{Dt}\mathrm{div}u\right\|_{L^2}\|\nabla\rho\|_{H^1}^2
                \right]\right\}ds\nonumber\\
                    &\leq& K(T)+\frac{\mu}{10}\int^t_0\left(
                    \|\nabla^2
                    \dot{u}\|_{L^2}^2+\|\nabla\frac{D}{Dt}\mathrm{div}u\|_{L^2}^2
                    \right)ds,
    \end{eqnarray}
        \begin{eqnarray}
          J_3&=&-\int^t_0\int
          \nabla(\rho u_j)\partial_j
          \dot{u}\nabla\dot{u}dxds\nonumber\\
            &\leq& C\int^t_0\|\nabla \dot{u}\|_{L^4}^2\left(
            \|\nabla\rho\|_{L^2}\|u\|_{L^\infty}+\|\nabla u\|_{L^2}
            \right)dxds\nonumber\\
                &\leq&K(T)+\frac{\mu}{10}\int^t_0
                    \|\nabla^2
                    \dot{u}\|_{L^2}^2ds,
        \end{eqnarray}
    \begin{eqnarray}
      J_4&=&\int^t_0\int\Delta\dot{u}^j[
            \partial_j P_t+\mathrm{div}(\partial_j
            Pu)]dxds\nonumber\\
                &=&-\int^t_0\int[\partial_j\Delta\dot{u}^jP'\rho_t
                +\partial_k\Delta\dot{u}^j\partial_jP
                u^k]dxds\nonumber\\
      &=&\int^t_0\int[
        P'\rho\mathrm{div}u\partial_j\Delta\dot{u}^j-\partial_k(\partial_j\Delta\dot{u}^ju^k)P
      +P\partial_j(\partial_k\Delta\dot{u}^ju^k)]dxds\nonumber\\
            &\leq&K\left(\int^t_0\int\left(|\nabla^2 u|+|\nabla \rho||\nabla
            u|\right)^2
            dxds    \right)^\frac{1}{2}
            \left(\int^t_0\int|\nabla^2\dot{u}|^2
            dxds    \right)^\frac{1}{2}\nonumber\\
                    &\leq& K(T)+\frac{\mu}{10}\int^t_0
                    \|\nabla^2
                    \dot{u}\|_{L^2}^2ds,
    \end{eqnarray}
        \begin{eqnarray}
          J_5&=&-\int^t_0\int\mu\Delta\dot{u}^j[\Delta u^j_t+\mathrm{div}(u\Delta
            u^j)]dxds\nonumber\\
                &=&\int^t_0\int\mu[
                \partial_i\Delta\dot{u}^j\partial_iu^j_t+\Delta u^j
                u\cdot\nabla\Delta \dot{u}^j
                ]dxds\nonumber\\
            &=&\int^t_0\int\mu[-
            |\nabla^2 \dot{u}|^2-\partial_i\Delta\dot{u}^j u^k\partial_k\partial_i
            u^j-\partial_i \Delta\dot{u}^j\partial_i u^k\partial_k u^j
            +\Delta u^ju\cdot\nabla\Delta\dot{u}^j
            ]dxds\nonumber\\
                                    &=&\int^t_0\int\mu[-
            |\nabla^2 \dot{u}|^2+\partial_i\Delta\dot{u}^j \mathrm{div}u\partial_i
            u^j-\partial_i \Delta\dot{u}^j\partial_i u^k\partial_k u^j
            -\partial_i u^j\partial_iu^k\partial_k\Delta\dot{u}^j
            ]dxds\nonumber\\
            &\leq&-\frac{1}{2}\int^t_0\int\mu|\nabla^2 \dot{u}|^2dxds
            +K\int^t_0\int |\nabla u|^2|\nabla^2 u|^2dxds\nonumber\\
                    &\leq& -\frac{1}{2}\int^t_0\int\mu|\nabla^2 \dot{u}|^2dxds+K(T),
        \end{eqnarray}
            \begin{eqnarray}
                J_6&=&-\int^t_0\int\Delta\dot{u}^j[\partial_j\partial_t((\lambda+\mu)\mathrm{div}u)
                +\mathrm{div}(u\partial_j((\lambda+\mu)\mathrm{div}u))]dxds\nonumber\\
            &=&\int^t_0\int\{\partial_j\Delta\dot{u}^j[\partial_t((\lambda+\mu)\mathrm{div}u)
                +\mathrm{div}(u(\lambda+\mu)\mathrm{div}u)]\nonumber\\
                    &&+\Delta\dot{u}^j\mathrm{div}(\partial_ju(\lambda+\mu)\mathrm{div}u)\}dxds\nonumber\\
                        &=&\int^t_0\int\partial_j\Delta\dot{u}^j[\partial_t((\lambda+\mu)\mathrm{div}u)
                +u^k\lambda'\partial_k\rho\mathrm{div}u+(\lambda+\mu)u^k\partial_k\mathrm{div}u]dxds+O_6\nonumber\\
                                &=&\int^t_0\int\partial_j\Delta\dot{u}^j[(\lambda+\mu)\frac{D}{Dt}\mathrm{div}u
                                +\lambda'\rho_t\mathrm{div}u
                +u^k\lambda'\partial_k\rho\mathrm{div}u]dxds+O_6\nonumber\\
                             &=&\int^t_0\int\partial_j\Delta(\partial_tu^j
                             +u\cdot\nabla u^j)(\lambda+\mu)\frac{D}{Dt}
                             \mathrm{div}udxds+O_6\nonumber\\
            &=&-\int^t_0\int(\lambda+\mu)|\nabla\frac{D}{Dt}
                             \mathrm{div}u|^2dxds+O_6+O_7\nonumber\\
            &=&-\int^t_0\int(\lambda+\frac{\mu}{    2})|\nabla\frac{D}{Dt}
                             \mathrm{div}u|^2dxds+K(T)+\frac{\mu}{10}\int^t_0\|\nabla^2\dot{u}\|_{L^2}^2ds,
            \end{eqnarray}
        \begin{eqnarray}
          J_7&=&-\int^t_0\int \Delta\dot{u}^j[(\rho f^j)_t+\mathrm{div}(u\rho f^j)]
           dxds\nonumber\\
                &=&-\int^t_0\int \Delta\dot{u}^j[\rho f^j_t+\rho u\cdot\nabla
                f^j]dxds\nonumber\\
                    &\leq& \frac{1}{10}\int^t_0\int\mu|\nabla^2 \dot{u}|^2dxds+K(T),\label{3VDD-full-E2.15-11}
        \end{eqnarray}
where $O_6$ denotes any term dominated by $C\int^t_0\int(|\nabla
\rho||\nabla u|^2+|\nabla u||\nabla^2 u|)|\nabla^2 \dot{u}|dxds$ and
$O_7$ denotes any term dominated by
$C\int^t_0\int(|\frac{D}{Dt}\mathrm{div}u||\nabla \rho|+|\nabla
u||\nabla^2 u|)|\nabla\frac{D}{Dt}\mathrm{div}u|+|\nabla u||\nabla^2
u||\nabla \rho||\frac{D}{Dt}\mathrm{div}u|dxds$,  $t\in[0,T]\cap
[0,T^*)$. From
(\ref{3VDD-full-E2.11-11})--(\ref{3VDD-full-E2.15-11}), we
immediately obtain (\ref{3VDD-full-E2.72-5}). Using  similar
arguments as that in the proof of Lemmas
\ref{3VDD-full-L2.7}--\ref{3VDD-full-L2.8}, we can easily get
(\ref{3VDD-full-E2.72-50}).
\end{proof}

Using the standard arguments based on the local existence results
together with the estimates (\ref{3VDD-full-E2.37-0}) and
(\ref{3VDD-full-E2.72-50}), we can obtain that $T^*=\infty$. Since
the uniqueness of the solution $(\rho-\tilde{\rho},u)\in
C([0,\infty);H^3)\cap C^1([0,\infty);H^2)$ is classical, we omit the
detail. Thus, we finish the proof of the existence and uniqueness
parts of Theorem \ref{3VDD-full-T1.1}.

\section{Large time Behavior}

From (\ref{3VDD-full-E2.16-4}), (\ref{3VDD-full-E2.22}) and
(\ref{3VDD-full-E2.28}), we have
    \begin{equation}
    \int^\infty_1\int(|\rho-\tilde{\rho}|^4+|F|^4)dxdt\leq
    C.\label{3VDD-full-E4.8}
     \end{equation}
From (\ref{3VDD-full-E2.25}), we have
    $$
    \int|\rho-\tilde{\rho}|^4(x,t)dx\leq
    \int|\rho-\tilde{\rho}|^4(x,s)dx+C\int^{N+2}_N\int|F|^4dxd\tau,
    $$
where $t\in[N+1,N+2]$ and $s\in[N,N+1]$, $N>1$. Integrating it with
$s$ in $[N,N+1]$, we obtain
      $$
    \sup_{t\in[N+1,N+2]}\int|\rho-\tilde{\rho}|^4(x,t)dx\leq
    C\int^{N+2}_N\int(|\rho-\tilde{\rho}|^4+|F|^4)dxd\tau.
    $$
Letting $N\rightarrow\infty$, using (\ref{3VDD-full-E4.8}), we can
easily obtain
    $$
\lim_{t\rightarrow+\infty}\int|\rho-\tilde{\rho}|^4(x,t)dx=0.
    $$

From  (\ref{3VDD-full-E2.31}), we can obtain
   \begin{equation}
    \int^\infty_1\int\left(|u|^4+|\nabla u|^4\right)dxds\leq
    C.\label{3VDD-full-E3.1-01}
    \end{equation}
From (\ref{3VDD-full-E2.1}), we have
    \begin{equation}
    \int^\infty_0\int|\nabla u|^2dxds\leq
    C.\label{3VDD-full-E3.1-02}
   \end{equation}
Thus, for all $\epsilon\in(0,1)$, there is a positive constant
$T_\epsilon$, such that for all $\tau>T_\epsilon$, we have
     \begin{equation}
    \int^\infty_\tau\int\left(|u|^4+|\nabla u|^4
    +|\nabla u|^2+|f|^4\right)dxds<\epsilon.\label{3VDD-full-E3.1-0}
     \end{equation}
For all $t>T_\epsilon+2$ and $\tau\in[t-1,t-2]$, from
(\ref{3VDD-full-E1.1}), (\ref{3VDD-full-E3.1-0})  and H\"{o}lder's
inequality, we get
    \begin{eqnarray*}
      &&\int\left[
          \frac{1}{4}\rho|u|^4\right](x,t)dx+\int^{t}_\tau\int|u|^2\left[
          \mu|\nabla
          u|^2+(\lambda+\mu)(\mathrm{div}u)^2\right]dxds\nonumber\\
            &=&\int\left[
          \frac{1}{4}\rho|u|^4
          \right](x,\tau)dx
          +\int^{t}_\tau\int\big[P\mathrm{div}(|u|^2u)-\frac{1}{2}\mu|\nabla |u|^2|^2-(\lambda+\mu)\mathrm{div}uu
          \cdot\nabla|u|^2
         +\rho f\cdot u|u|^2\big]dxdt\nonumber\\
            &\leq &\int\left[
          \frac{1}{4}\rho|u|^4
          \right](x,\tau)dx
          +C\left(\int^{t}_\tau\int|\nabla u|^2dxds\right)^\frac{1}{2}
          \left(\int^{t}_\tau\int|u|^4dxds\right)^\frac{1}{2}\nonumber\\
            &&
                    +C\left(\int^{t}_\tau\int|\nabla u|^4dxds\right)^\frac{1}{2}
          \left(\int^{t}_\tau\int|u|^4dxds\right)^\frac{1}{2}
          +C\left(\int^{t}_\tau\int|f|^4dxds\right)^\frac{1}{4}
          \left(\int^{t}_\tau\int|u|^4dxds\right)^\frac{3}{4}\nonumber\\
            &\leq &\int\left[
          \frac{1}{4}\rho|u|^4
          \right](x,\tau)dx
          +C\epsilon.
    \end{eqnarray*}
Integrating it with $\tau$ in $[t-1,t-2]$, we obtain
    $$
      \int\left[
          \frac{1}{4}\rho|u|^4
          \right](x,t)dx
            \leq \int^{t-1}_{t-2}\int\left[
          \frac{1}{4}\rho|u|^4
          \right](x,\tau)dxd\tau+C\epsilon
            \leq C\epsilon.
    $$
Thus, we immediately obtain (\ref{3VDD-full-E1.15-10}).

From (\ref{3VDD-full-E2.64-0}), we have
    \begin{equation}
     \|\dot{u}\|_{L^2}\leq C\|\sqrt{\rho}\dot{u}\|_{L^2}
     +C\|\nabla\dot{u}\|_{L^2}.\label{3VDD-full-E3.5}
    \end{equation}
From   (\ref{3VDD-full-E2.22}),
(\ref{3VDD-full-E3.1-01})--(\ref{3VDD-full-E3.1-02}) and
(\ref{3VDD-full-E3.5}), we have that for all $\epsilon\in(0,1)$,
there is   a positive constant $T_{2\epsilon}$, such that for all
$\tau>T_{2\epsilon}$, we have
     \begin{equation}
  \int^\infty_\tau\int\left(|u|^4
    +|\nabla u|^4+|\nabla u|^2+|\dot{u}|^2+|\nabla \dot{u}|^2\right)dxds<\epsilon.
     \end{equation}
 For all $t>T_{2\epsilon}+2$
and $\tau\in[t-1,t-2]$, multiplying (\ref{3VDD-full-E1.1})$_2$ by
$\dot{u}$, integrating it over $\mathbb{R}^N\times[\tau,t]$, we
obtain
    \begin{eqnarray}
      &&\int^t_\tau\int\rho|\dot{u}|^2dxds\nonumber\\
            &=&\int^t_\tau\int\left(-\dot{u}\cdot\nabla P+\mu\Delta
            u\cdot\dot{u}+\nabla((\lambda+\mu)\mathrm{div}u)\cdot\dot{u}
            +\rho f\cdot\dot{u}
            \right)dxds\nonumber\\
                &:=&\sum^4_{i=1}J_i.\label{3VDD-full-E2.6-20}
    \end{eqnarray}
Using the integration by parts and H\"{o}lder's inequality, we have
    \begin{eqnarray}
      J_1&=&-\int^t_\tau\int\dot{u}\cdot\nabla P dxds
      \nonumber\\
            &=&\int^t_\tau\int \mathrm{div}\dot{u}
            (P-P(\tilde{\rho}))
           dxds      \nonumber\\
        &\leq& C\left[\int^t_\tau\int|\nabla
        \dot{u}|^2dxds
        \right]^\frac{1}{2}  \nonumber\\
        &\leq& C \epsilon^\frac{1}{2},
    \end{eqnarray}
        \begin{eqnarray}
          J_2&=&\int^t_\tau\int\mu \Delta
          u\cdot\dot{u}dxds\nonumber\\
                &=&-\frac{\mu}{2}\int|\nabla u|^2(x,t)dx
                +\frac{\mu}{2}\int|\nabla u|^2(x,\tau)dx
                               -\int^t_\tau\int\mu\partial_i u^j\partial_i(u^k\partial_k
                u^j)dxds\nonumber\\
          &\leq&-\frac{\mu}{2}\int|\nabla u|^2(x,t)dx
                +\frac{\mu}{2}\int|\nabla u|^2(x,\tau)dx
            +C \int^t_\tau\int|\nabla u|^3dxds
           \nonumber\\
          &\leq&-\frac{\mu}{2}\int|\nabla u|^2(x,t)dx
                +\frac{\mu}{2}\int|\nabla u|^2(x,\tau)dx
            +C \epsilon,
        \end{eqnarray}
            \begin{eqnarray}
              J_3&=&\int^t_\tau\int\nabla((\lambda+\mu)\mathrm{div}u)\cdot\dot{u}dxds\nonumber\\
              &=&-\frac{1}{2}\int\left[(\lambda+\mu)|\mathrm{div}u|^2
              \right](x,t)dx+\frac{1}{2}\int\left[(\lambda+\mu)|\mathrm{div}u|^2
              \right](x,\tau)dx
              +\int^t_\tau\int\frac{1}{2}\lambda'\rho_t|\mathrm{div}u|^2dxds\nonumber\\
                    &&-\int^t_\tau\int(\lambda+\mu)\mathrm{div}u\mathrm{div}(u\cdot\nabla u)dxds
                                       \nonumber\\
                        &\leq&-\frac{1}{2}\int\left[(\lambda+\mu)|\mathrm{div}u|^2
              \right](x,t)dx+\frac{1}{2}\int\left[(\lambda+\mu)|\mathrm{div}u|^2
              \right](x,\tau)dx+C \int^t_\tau\int|\nabla
              u|^3dxds              \nonumber\\
                        &\leq&-\frac{1}{2}\int\left[(\lambda+\mu)|\mathrm{div}u|^2
              \right](x,t)dx+\frac{1}{2}\int\left[(\lambda+\mu)|\mathrm{div}u|^2
              \right](x,\tau)dx+C \epsilon,
            \end{eqnarray}
        \begin{equation}
          J_4=\int^t_\tau\int\rho f\cdot\dot{u} dxds\leq
          C\left(\int^t_\tau\int |\dot{u}|^2
          dxds\right)^\frac{1}{2}\leq C\epsilon^\frac{1}{2}.\label{3VDD-full-E2.10-20}
        \end{equation}
From (\ref{3VDD-full-E2.1}),
(\ref{3VDD-full-E2.6-20})--(\ref{3VDD-full-E2.10-20}), we obtain
    $$
      \frac{\mu}{2}\int|\nabla u|^2(x,t)dx
      +\int^t_\tau\int\rho|\dot{u}|^2dxds
            \leq C \epsilon^\frac{1}{2}+C\int|\nabla u|^2(x,\tau)dx
    $$
Integrating it with $\tau$ in $[t-1,t-2]$, we obtain
    $$
      \frac{\mu}{2}\int|\nabla u|^2(x,t)dx
            \leq C \epsilon^\frac{1}{2}+C\int^{t-1}_{t-2}\int|\nabla u|^2(x,\tau)dx
                       \leq C \epsilon^\frac{1}{2}.
    $$
Thus, we immediately obtain (\ref{3VDD-full-E1.15-20}).

 Thus, we
finish the proof of  Theorem \ref{3VDD-full-T1.1}.

\section{Proof of Theorem \ref{2VDD-T1.1}}\label{2VDD-S3}

Let $j_\delta(x)$ be a standard mollifying kernel of width $\delta$.
Define the approximate initial data $(\rho_0^\delta,u_0^\delta)$ by
    $$
    \rho_0^\delta=j_\delta*\rho_0+\delta,\
        u_0^\delta=j_\delta*u_0.
    $$
Assuming that similar smooth approximations have been constructed
for functions $P$, $f$ and $\lambda$, we may then apply Theorem
\ref{3VDD-full-T1.1} to obtain a global smooth solution
$(\rho^\delta,u^\delta)$ of
(\ref{3VDD-full-E1.1})--(\ref{3VDD-full-E1.2}) with the initial data
$(\rho_0^\delta,u_0^\delta)$, satisfying the bound estimates of
Propositions \ref{3VDD-full-P2.1}-\ref{3VDD-full-P2.3} with
constants independent of $\delta$.

First, we obtain the strong limit of $\{u^\delta\}$. From
(\ref{3VDD-full-E2.22}) and (\ref{3VDD-full-E2.35}), we have
    \begin{equation}
      <u^\delta(\cdot,t)>^\alpha\leq C(\tau),\
     \ t\geq\tau>0,\label{2VDD-E3.1}
    \end{equation}
where $\alpha\in(0,1)$ when $N=2$, $\alpha\in(0,\frac{1}{2}]$ when
$N=3$.
 From (\ref{2VDD-E3.1}), we have
    $$
    \left|u^\delta(x,t)-\frac{1}{|B_R(x)|}\int_{B_R(x)}u^\delta(y,t)dy
    \right|\leq C(\tau)R^\alpha,\ t\geq\tau>0.
    $$
Taking $R=1$, from (\ref{3VDD-full-E2.16-0}) and
(\ref{3VDD-full-E2.22}), we have
    \begin{equation}
      \|u^\delta\|_{L^\infty(\mathbb{R}^N\times[\tau,\infty))}\leq
      C(\tau).\label{2VDD-E3.2}
    \end{equation}
Then, we need only to derive a modulus of H\"{o}lder continuity in
time.
      For all $t_2\geq
t_1\geq\tau$, from (\ref{3VDD-full-E2.1}), (\ref{3VDD-full-E2.22}),
(\ref{3VDD-full-E2.64-0})
 and (\ref{2VDD-E3.2}), we have
    \begin{eqnarray*}
      &&|u^\delta(x,t_2)-u^\delta(x,t_1)|\nonumber\\
            &\leq&\frac{1}{|B_R(x)|}\int^{t_2}_{t_1}\int_{B_R(x)}
            |u^\delta_t(y,s)|dyds+C(\tau)R^\alpha\nonumber\\
      &\leq&CR^{-\frac{N}{2}}|t_2-t_1|^\frac{1}{2}\left(\int^{t_2}_{t_1}\int
            |u^\delta_t|^2dyds\right)^\frac{1}{2}+C(\tau)R^\alpha\nonumber\\
      &\leq&CR^{-\frac{N}{2}}|t_2-t_1|^\frac{1}{2}\left(\int^{t_2}_{t_1}\int
            |\dot{u}^\delta|^2+|u^\delta\cdot\nabla u^\delta|^2dyds\right)^\frac{1}{2}+C(\tau)R^\alpha\nonumber\\
      &\leq&C(\tau)(R^{-\frac{N}{2}}|t_2-t_1|^\frac{1}{2}+R^\alpha).
        \end{eqnarray*}
Choosing $R=|t_2-t_1|^{\frac{1}{N+2\alpha}}$, we have
   \begin{equation}
    <u^\delta>^{\alpha,\frac{\alpha}{N+2\alpha}}_{\mathbb{R}^N\times[\tau,\infty)}\leq
    C(\tau),\ \tau>0.\label{2VDD-E3.3}
    \end{equation}
From the Ascoli-Arzela theorem, we have (extract a subsequence)
    \begin{equation}
      u^\delta\rightarrow u,
      \ \textrm{uniformly on compact sets in
      }\mathbb{R}^N\times(0,\infty).\label{2VDD-E3.4}
    \end{equation}

Second, we obtain the strong limits of $\{F^\delta\}$ and
$\{w^\delta\}$. From
(\ref{3VDD-full-E2.16-2})--(\ref{3VDD-full-E2.16-3}),
(\ref{3VDD-full-E2.22}) and (\ref{3VDD-full-E2.66}), using similar
arguments as that in the proof of
(\ref{2VDD-E3.1})--(\ref{2VDD-E3.2}), we have
    \begin{equation}
      <F^\delta(\cdot,t)>^{\alpha'}+\|F^\delta\|_{L^\infty(\mathbb{R}^N\times[\tau,T])}
      + <w^\delta(\cdot,t)>^{\alpha'}+\|w^\delta\|_{L^\infty(\mathbb{R}^N\times[\tau,T])}
      \leq C(\tau,T),\label{2VDD-E3.5}
    \end{equation}
where $0<\tau\leq t\leq T$ and ${\alpha'}\in(0,\frac{2+q-N}{2+q}]$.
The simple computation implies that
    \begin{eqnarray}
      F^\delta_t&=&
      \rho^\delta(2\mu+\lambda(\rho^\delta))\left(
      F^\delta\frac{d}{ds}\left.\left(\frac{1}{2\mu+\lambda(s)}\right)\right|_{s=\rho^\delta}
      +\frac{d}{ds}\left.\left(\frac{P(s)-P(\tilde{\rho})}{2\mu+\lambda(s)}\right)\right|_{s=\rho^\delta}
      \right)\mathrm{div}u^\delta\nonumber\\
        &&-u^\delta\cdot\nabla F^\delta+(2\mu+\lambda(\rho^\delta))\mathrm{div}\dot{u}^\delta-(2\mu+\lambda(\rho^\delta)
        )\partial_iu^\delta_j\partial_ju^\delta_i
    \end{eqnarray}
and
    \begin{equation}
    (w^\delta)^{k,j}_t=-u^\delta\cdot\nabla (w^\delta)^{k,j}
    +\partial_j \dot{u}^\delta_k
    -\partial_k \dot{u}^\delta_j
    -\partial_ju^\delta_i\partial_iu^\delta_k
    +\partial_ku^\delta_i\partial_iu^\delta_j.
    \end{equation}
Then, from (\ref{3VDD-full-E2.16-2}), (\ref{3VDD-full-E2.22}),
(\ref{3VDD-full-E2.31}), (\ref{2VDD-E3.2}) and (\ref{2VDD-E3.5}), we
have
    $$
    \|F^\delta_t\|_{L^2(\mathbb{R}^N\times[\tau,T])}
    +\|w^\delta_t\|_{L^2(\mathbb{R}^N\times[\tau,T])}\leq
    C(\tau,T),\ T>\tau>0.
    $$
Using a similar argument as that in the proof of (\ref{2VDD-E3.3}),
we obtain
    \begin{equation}
    <F^\delta>^{{\alpha'},\frac{{\alpha'}}{N+2{\alpha'}}}_{\mathbb{R}^N\times[\tau,T]}+
     <w^\delta>^{{\alpha'},\frac{{\alpha'}}{N+2{\alpha'}}}_{\mathbb{R}^N\times[\tau,T]}\leq
    C(\tau,T),\ T>\tau>0.
    \end{equation}
and (extract a subsequence)
        \begin{equation}
      F^\delta\rightarrow F,
      \  w^\delta\rightarrow w,
      \ \textrm{uniformly on compact sets in
      }\mathbb{R}^N\times(0,\infty).
    \end{equation}

    Third, we obtain the strong limit of $\{\rho^\delta\}$. From
    (\ref{3VDD-full-E2.37-0}), we get (extract a subsequence)
    $$
    \rho^\delta\overset{*}{\rightharpoonup}\rho,
    \ \textrm{ weak-* in    }\ L^\infty(\mathbb{R}^N).
    $$
Let $\Phi(s)$ be an arbitrary continuous function on
$[0,\bar{\rho}]$. Then, we have that (extract a subsequence)
$\Phi(\rho^\delta)$ converges weak-$*$ in $L^\infty(\mathbb{R}^N)$.
Denote the weak-$*$ limit by $\bar{\Phi}$:
    $$
     \Phi(\rho^\delta)\overset{*}{\rightharpoonup}\bar{\Phi},
    \ \textrm{ weak-* in    }\ L^\infty(\mathbb{R}^N).
    $$
From the definition of $F$, we have
    \begin{equation}
      \mathrm{div}u=\bar{\nu}F+\overline{P_0},
    \end{equation}
where
    $$
    \nu(\rho)=\frac{1}{2\mu+\lambda(\rho)},
    \ P_0(\rho)=\nu(\rho)(P(\rho)-P(\tilde{\rho})).
    $$
From (\ref{3VDD-full-E1.1}), we have
    $$
    \partial_t\overline{\rho\ln\rho}+\mathrm{div}(\overline{\rho\ln\rho}
    u)+F\overline{\rho\nu}+\overline{\rho P_0}=0
    $$
and
    $$
    \partial_t(\rho\ln\rho)+\mathrm{div}({\rho\ln\rho}
    u)+F\rho\overline{\nu}+\rho\overline{ P_0}=0.
    $$
Letting $\Psi=\overline{\rho\ln\rho}-{\rho\ln\rho}\geq0$, we obtain
    \begin{equation}
      \partial_t\Psi+ +\mathrm{div}(\Psi
    u)+F(\overline{\rho\nu}-\rho\nu)
    +F\rho(\nu-\bar{\nu})+\overline{\rho P_0}-\rho\overline{
    P_0}=0.\label{2VDD-E3.11}
    \end{equation}
with the initial condition $\Psi|_{t=0}=0$ almost everywhere in
$\mathbb{R}^N$. Let $\phi(s)=s\ln s$. Since
    $$
    \phi''(s)=\frac{1}{s}\geq
\frac{1}{\bar{\rho}},
    \ s\in[0,\bar{\rho}],
    $$
we get
    $$
    \phi(\rho^\delta)-\phi(\rho)
    =\phi'(\rho)(\rho^\delta-\rho)
    +\frac{1}{2}\phi''(\rho+\xi(\rho^\delta-\rho))(\rho^\delta-\rho)^2,
    \ \xi\in[0,1],
    $$
and
    \begin{equation}
      \overline{\lim_{\delta\rightarrow0}}\|\rho^\delta-\rho\|_{L^2}^2
      \leq C\|\Psi\|_{L^1}.
    \end{equation}
Similarly, every function $f\in C^2([0,\bar{\rho}])$ satisfies
    \begin{equation}
      \left|\int g(\bar{f}-f(\rho))dx
      \right|\leq C\int |g|\Psi dx,
    \end{equation}
where $g$ is any function such that the integrations exist. Then,
when $\nu\in C^2([0,\bar{\rho}])$, we have
    \begin{equation}
      \left|\int F(\overline{\rho\nu}-\rho\nu)dx
      \right|\leq C\int |F|\Psi dx
    \end{equation}
and
    \begin{equation}
      \left|\int F\rho(\bar{\nu}-\nu)dx
      \right|\leq C\int |F|\Psi dx.
    \end{equation}
When $P_0\in C^2([0,\bar{\rho}])$, we have
    \begin{equation}
    \left|\int (\overline{\rho P_0}-\rho \overline{P_0})dx
      \right|\leq    \left|\int (\overline{\rho P_0}-\rho P_0)dx
      \right|+\left|\int \rho(\overline{P_0}-P_0)dx
      \right|\leq C\int \Psi dx.
    \end{equation}
When $P_0$ is monotone function on $[0,\bar{\rho}]$, using the Lemma
5 in \cite{Vaigant}, we have
     \begin{equation}
     \overline{\rho P_0}\geq \rho \overline{P_0}.\label{2VDD-E3.18}
    \end{equation}
From (\ref{2VDD-E3.11})--(\ref{2VDD-E3.18}), we obtain
    $$
    \int\Psi dx\leq \int^t_0\int(1+ |F|)\Psi dxds.
    $$
Using (\ref{3VDD-full-E2.72}) and Gronwall's inequality, we get
    $$
    \Psi=0,\ (t,x)\in[0,T]\times\mathbb{R}^N,
    $$
and  (extract a subsequence)
    $$
    \rho^\delta-\tilde{\rho}\rightarrow \rho-\tilde{\rho},
    \ \textrm{ strongly in } L^k(\mathbb{R}^N\times[0,\infty)),
    $$
for all $k\in[2,\infty)$.

Thus, it is easy to show that the limit function $(\rho,u)$ are
indeed a weak solution of the system
(\ref{3VDD-full-E1.1})--(\ref{3VDD-full-E1.2}). Using a similar
argument as that in the proof of (\ref{3VDD-full-E1.15-10}), we get
(\ref{2VDD-full-E1.28}). This finishes the proof of Theorem
\ref{2VDD-T1.1}. {\hfill $\square$\medskip}

\section{Appendix}\label{3VDD-full-A1}
It requires a dyadic decomposition of the Fourier space, so let us
start by recalling the definition of the following operators of
localization in Fourier space (see \cite{Chemin}):
   $$
   \Delta_{q}a \triangleq
   \mathcal{F}^{-1}(\varphi(2^{-q}|\xi|)\hat{a}),\ \ \ \ {\rm for}\ q\in
   \mathbb{Z},
    $$
where $\mathcal{F}a$ and $\hat{a}$ denote the Fourier transform of
any function $a$. The function $\varphi$ is smooth, and satisfies
    $$
    \mathrm{supp}
\varphi\in\left\{\xi\in\mathbb{R}^2\left||\xi|\in[\frac{3}{4},\frac{8}{3}]\right.\right\},
\ \sum_{j\in \mathbb{Z}}\varphi(2^{-j}t)=1, \ \forall\ t\in
\mathbb{R}\backslash \{0\}.
    $$
 Let us note that if $|j-j'|\geq
5$, then  $\mathrm{supp}\varphi(2^{-j}t)\cap \mathrm{supp}
\varphi(2^{-j'}t)=\emptyset $.

\begin{defn} We denote by $\dot{B}_{p,q}^{s}$ the space of
distributions, which is the completion of
$\mathcal{S}(\mathbb{R}^N)$, $N\geq2$, by the following norm:
$$\|a\|_{\dot{B}_{p,q}^{s}}\stackrel {\mathrm {def}}{=} \left\|
2^{sk}\|\Delta_{k}a\|_{L^{p}(\mathbb{R}^N)} \right\|_{l^q_{k}}.$$
\end{defn}

\begin{lem}[\cite{Chemin}]
Denote $\mathcal{B}$ a ball of $\mathbb{R}^N$, and $\mathcal{C}$ a
ring of $\mathbb{R}^N$. Assume that $1\leq p_{2}\leq p_{1}\leq
\infty$ and $1\leq q_{2}\leq q_{1}\leq \infty$. If the support of
$\hat{a}$ is included in $2^{k}\mathcal{B}$, then
$$\|\partial^{\alpha}a\|_{L^{p_{1}}}\lesssim 2^{k(|\alpha|+N(\frac{1}{p_{2}}-\frac{1}{p_{1}}))}
\|a\|_{L^{p_{2}}}.$$ If the support of $\hat{a}$ is included in
$2^{k}\mathcal{C}$, then
$$\|a\|_{L^{p_{1}}}\lesssim
2^{-kM}\sup_{|\alpha|=M} \|\partial^{\alpha}a\|_{L^{p_{1}}}.$$
\end{lem}

\begin{lem} For any $p>N$ and $N\geq2$, there exists a positive constant $C_p$
such that
  \begin{equation}
      \|\nabla u\|_{L^\infty}\leq C\|u\|_{L^2}+C_p(\|\nabla
      u\|_{\dot{B}^0_{\infty,\infty}}+1)\log\left(
      e+\frac{\|\Delta u\|_{L^{p}}}{\|\nabla u\|_{\dot{B}^0_{\infty,\infty}}+1}
      \right).\label{2VDD-full-E3.1}
    \end{equation}
\end{lem}
\begin{proof}
\begin{eqnarray*}
        \|\nabla u\|_{L^\infty}&\leq&C\|u\|_{L^2}+\sum_{k\geq0}\|\nabla\Delta_k
        u\|_{L^\infty}=C\|u\|_{L^2}+\sum_{k\geq0}^{M-1}\|\nabla\Delta_k
        u\|_{L^\infty}+\sum_{k\geq M}\|\nabla\Delta_k
        u\|_{L^\infty}\\
            &\leq& C\|u\|_{L^2}+M\|\nabla u\|_{\dot{B}^{0}_{\infty,\infty}}
            +\sum_{k\geq M}2^{k(\frac{N}{p}-1)}\|\Delta \Delta_ku\|_{L^{p}}\\
            &\leq& C\|u\|_{L^2}+M\|\nabla u\|_{\dot{B}^{0}_{\infty,\infty}}
            +\frac{C2^{M(\frac{N}{p}-1)}}{1-2^{\frac{N}{p}-1}}\|\Delta
            u\|_{L^{p}}.
\end{eqnarray*}
Choosing $M\sim\frac{p}{p-N} \log\left(
      e+\frac{\|\Delta u\|_{L^{p}}}{\|\nabla u\|_{\dot{B}^0_{\infty,\infty}}+1}
      \right)$, we can finish the proof.
\end{proof}

\section*{Acknowledgment}

The author would like to thank Professor Daoyuan Fang for his
reading the draft of this paper and valuable discussions. The author
was partially supported by the National Science Foundation of China
under grants 10871175 and a Project Supported by Scientific Research
Fund of Zhejiang Provincial Education Department under grants
Y200803203.

\end{document}